\documentclass[12pt,a4paper,  reqno, english]{amsart}
\usepackage[utf8]{inputenc}
\usepackage[margin=0.9in]{geometry} 
\usepackage{amsthm}
\usepackage{amsmath}
\usepackage{amsfonts}
\usepackage{amssymb}
\usepackage{chngcntr}
\usepackage{url}
\usepackage{hyperref}
\usepackage[makeroom]{cancel}
\usepackage{enumerate}
\usepackage{mathrsfs}
\usepackage{verbatim}
\newtheorem{Theorem}{Theorem}

\newtheorem{Proposition}[Theorem]{Proposition}
\newtheorem{Corollary}[Theorem]{Corollary}
\newtheorem{Lemma}[Theorem]{Lemma}

\newtheorem{Definition}[Theorem]{Definition}
\newtheorem{Remark}[Theorem]{Remark}

\newcommand{\dist}{\operatorname{dist}}

\newcommand{\vol}{\operatorname{vol}}
\newcommand{\spn}{\operatorname{span}}
\newcommand{\rank}{\operatorname{rank}}

\newcommand{\ex}{\operatorname{ex}}
\newcommand{\lex}{\operatorname{lex}}
\newcommand{\unif}{\operatorname{unif}}
\newcommand{\bas}{\operatorname{bas}}
\newcommand{\res}{\operatorname{res}}
\newcommand{\abs}{\operatorname{abs}}
\newcommand{\sign}{\operatorname{sign}}
\newcommand{\Str}{\operatorname{Str}}
\newcommand{\Kh}{\operatorname{Kh}}
\counterwithin{Theorem}{section}
\numberwithin{equation}{section}
\title{Effective results on the size and structure of sumsets} 
\author{Andrew Granville}
\thanks{A.G. is funded by the Natural Sciences and Engineering Research Council of Canada (NSERC) under the Canada Research Chairs program.} 
\author{George Shakan}
\thanks{G.S. is supported by Ben Green's Simons Investigator Grant 376201.}
 \author{Aled Walker}
\thanks{A.W. was supported by a postdoctoral research fellowship at the Centre de Recherches Mathématiques and a junior fellowship at Institut Mittag-Leffler, and is  a Junior Research Fellow at Trinity College Cambridge.} 

\begin{document}

\begin{abstract}
Let $A \subset \mathbb{Z}^d$ be a finite set. It is known that $NA$ has a particular size ($\vert NA\vert = P_A(N)$ for some $P_A(X) \in \mathbb{Q}[X]$)
and structure (all of the lattice points in a cone other than certain exceptional sets), once $N$ is larger than some threshold.  In this article we give the first effective upper bounds for this threshold for arbitrary $A$.   Such explicit results  were only previously known in the special cases when $d=1$,
when the convex hull of $A$ is a simplex or when $\vert A\vert = d+2$ \cite{CG20},  results which   we improve.
\end{abstract}
\maketitle
 

\section{Introduction}
For any given finite subset $A$ of an abelian group $G$, we consider the sumset 
\[NA : = \{a_1 + a_2 + \cdots + a_N: a_i \in A\}.\] 
If $G$ is finite 
and $N$ is sufficiently large then 
\begin{equation}
\label{eq finite group case}
NA = Na_0 + \langle A-A\rangle
\end{equation}
for any $a_0 \in A$ where $\langle A-A\rangle$ is the subgroup of $G$ generated by $A-A$, so that $\vert NA\vert$ is eventually constant. 
In this article we study instead the case when $G= \mathbb{Z}^d$ is infinite, and ask similar questions about the size and structure of $NA$ when $N$ is large. 

\subsection*{The size of $NA$}   Khovanskii's 1992 theorem \cite{Kh92}  states  that if $A \subset \mathbb{Z}^d$ is   finite   then there exists  $P_A(X) \in \mathbb{Q}[X]$ of degree  $\leqslant d$ such that  if $N\geqslant N_{\text{Kh}}(A)$  then
\[
\vert N A\vert = P_A(N).
\]
Although there are now several different proofs of Khovanskii's theorem  \cite{NR02, JK08},  the only effective bounds on $N_{\text{Kh}}(A)$ have been obtained when $d=1$ \cite{Na72,WCC11,GS20,GW20},  when the convex hull of $A$ is a $d$-simplex  or when $\vert A\vert = d+2$ (see \cite{CG20}). 

We will determine an upper bound for $N_{\text{Kh}}(A)$  for any such  $A$ in terms of the \emph{width} of $A$,
\begin{equation}
\label{eq width of A}
w(A) =\text{width}(A):= \max_{a_1,a_2 \in A} \Vert a_1 - a_2 \Vert_\infty .
\end{equation}

\begin{Theorem}[Effective Khovanskii]
\label{thm:effectiveKhovanksii}
If $A \subset \mathbb{Z}^d$ is finite then 
\[
\vert NA\vert = P_A(N) \text{ for all } N \geqslant  
(2\vert A\vert\cdot \text{\rm width}(A))^{(d+4)\vert A\vert}.
\] 
\end{Theorem}

The theorem states that $N_{\text{Kh}}(A)\leqslant (2\ell\, w(A))^{(d+4)\ell}$ where $\ell:=|A|$. We expect that $N_{\text{Kh}}(A)$ is considerably smaller (see Section \ref{1-d}); for example, if 
$|A|=d+2$ and $A-A$ generates $\mathbb Z^d$ then \cite[Theorem 1.2]{CG20} gives that 
\begin{equation}
\label{eq CGbound}
N_{\operatorname{Kh}}(A)= d! \text{ Vol}(H(A)) -d-1 ,
\end{equation} 
where the   \emph{convex hull} $H(A)$  is defined by 
\[ 
H(A) := \bigg\{\sum\limits_{a \in A} c_a a: \text{Each } c_a \in \mathbb{R}_{ \geqslant 0} , \, \sum\limits_{a \in A} c_a = 1\bigg\}   .
\]


We can replace $w(A)$ in Theorem \ref{thm:effectiveKhovanksii} by $w^*(A)$ which is defined to be the minimum of $w(A')$ over all 
 $A^\prime \subset \mathbb{Z}^d$ that are Freiman isomorphic to $A$.\footnote{That is, there is a map $\phi:A\to A'$ such that for all $a_1,\dots,a_k,b_1,\dots ,b_k\in A$ and $k\geqslant 1$, we have
\[
a_1+\dots+a_k=b_1+\dots +b_k  \text{ if and only if } \phi(a_1)+\dots+\phi(a_k)=\phi(b_1)+\dots +\phi(b_k ).
\]
}
  
Previous proofs of Khovanskii's theorem \cite{NR02, JK08} relied on the following ineffective principle. 

\begin{Lemma}[The Mann-Dickson  Lemma]
\label{Lemma Mann Dickson}
For any $S \subset \mathbb{Z}_{\geqslant 0}^d$  there exists a finite subset $S_{\min} \subset S$ such that for all $s \in S$ there exists $x \in S_{\min}$ with $s - x \in \mathbb{Z}_{ \geqslant 0}^d$. 
\end{Lemma} 

\noindent For a proof see \cite[Lemma 5]{GS20}. Here we rework   the method of Nathanson--Ruzsa from \cite{NR02} as a collection of linear algebra problems which we solve quantitatively (see Section \ref{6}), and therefore bypass Lemma \ref{Lemma Mann Dickson} and prove our effective threshold.

\subsection*{The structure of $NA$}  For a given finite set $A \subset \mathbb{Z}^d$ with $0\in A$ we have
\[ 
H(A)  = \bigg\{\sum\limits_{a \in A} c_a a:  \text{Each } c_a \in \mathbb{R}_{ \geqslant 0} ,   \, \sum\limits_{a \in A} c_a \leqslant 1\bigg\} .
\]
We let  $\ex(H(A))$ be the set of extremal points of $H(A)$, that is the ``corners'' of the boundary of $A$, \footnote{That is, those points $p \in H(A)$ for which there is a vector $v \in \spn(A-A)\setminus \{0\}$ and a constant $c$ such that $\langle v , p \rangle = c$ and $\langle v , x \rangle > c$ for all $x \in H(A) \setminus\{p\}$; see Appendix \ref{Appendix convex sets}.} which is a subset of $A$.
We define the lattice generated by $A$,
\[
\Lambda_A: = \bigg\{ \sum\limits_{a \in A} x_a a: x_a \in \mathbb{Z} \text{ for all } a\bigg\}  .
\]
   For a domain $D \subset \mathbb{R}^d$ we set $N \cdot D:=\{Nx: x \in D\}$  so that $N \cdot H(A) = NH(A)$ as  $H(A)$ is convex and so, as $0\in A$,
\[
H(A) \subset 2 H(A)\subset 3 H(A) \dots \subset C_A := \lim_{N\to \infty} NH(A)= \bigg\{ \sum\limits_{ a \in A} c_a a: c_a \in \mathbb{R}_{\geqslant 0} \text{ for all } a\bigg\} ,
\]
the cone generated by $A$. Now, by definition,
\[
0\in A\subset 2A\subset 3A \dots \subset \mathcal{P}(A) := \bigcup \limits_{N=1}^\infty NA,
\]
and each 
\[ NA \subset N H(A) \cap \Lambda_A \text { so that }  \mathcal{P}(A) \subset C_A \cap \Lambda_{A}.
\] Define the set of \emph{exceptional elements}
 \[ 
 \mathcal{E}(A): = (C_A \cap \Lambda_{A}) \setminus \mathcal{P}(A). 
 \]  
 
 Therefore, for any finite $A \subset \mathbb{Z}^d$ and $a \in A$ we have 
 \[  
 N(a-A) \subset ( N H(a-A) \cap \Lambda_{a - A}) \setminus  \mathcal{E}(a-A),
 \]
 as $0 \in a-A$. So\[NA \subset ( N H(A) \cap (aN +\Lambda_{a-A})) \setminus (aN- \mathcal{E}(a-A)).\] Hence, as $aN + \Lambda_{a-A}$ is independent of the choice of $a \in A$ and $\Lambda_{a-A} = \Lambda_{A-A}$, for any fixed $a_0 \in A$ we have
 \begin{equation}
\label{eq:thirdinclusion}
NA \subset (N H(A) \cap (a_0N + \Lambda_{A-A}) ) \setminus \Big( \bigcup \limits_{a \in \ex(H(A))} (aN - \mathcal{E}(a-A))\Big) .
\end{equation}

 In  \cite{GS20} the first two authors showed\footnote{The result in \cite{GS20} was only stated when $0 \in A$ and $\Lambda_A = \mathbb{Z}^d$, and the union was over all of $A$ rather than just $\ex(H(A))$, but the methods give the general version \eqref{eq extremal union}.} there exists a constant $N_{\text{Str}}(A)$ such that we get equality in 
\eqref{eq:thirdinclusion} provided $N\geqslant N_{\text{Str}}(A)$; that is,
\begin{equation}
\label{eq extremal union}
 NA = (NH(A) \cap (a_0N + \Lambda_{A-A})) \setminus \Big( \bigcup\limits_{a\in \ex(H(A))}(aN - \mathcal{E}(a-A))\Big).
 \end{equation} 
(Compare this statement to \eqref{eq finite group case}.) The proof in \cite{GS20} relied on the ineffective Lemma \ref{Lemma Mann Dickson} so did not produce a  value for $N_{\text{Str}}(A)$. 

In this article we give an effective bound on $N_{\text{Str}}(A)$:

\begin{Theorem}[Effective structure]
\label{thm main theorem}
If $A \subset \mathbb{Z}^d$ is finite then 
 \[   \eqref{eq extremal union} \text{ holds for all } N \geqslant (  d\vert A\vert \cdot \text{\rm width}(A))^{13d^6}.\] 
\end{Theorem}
\noindent That is, Theorem \ref{thm main theorem} implies that $N_{\text{Str}}(A)\leqslant ( d\ell\, w(A))^{13d^6}$ where $\vert A\vert = \ell$.

The 1-dimensional case is easier than higher dimensions, since if $0 = \min A$ and $\Lambda_{A}=\mathbb{Z}$ then  $ \mathcal{E}(A)$ is finite, and so
has been the subject of much study \cite{Na72, WCC11, GS20, GW20}: We have $N_{\text{Str}}(A)=1$ if $\vert A\vert =3$ in \cite{GS20}, and $N_{\text{Str}}(A)\leqslant w(A)+2-\vert A\vert $ in  \cite{GW20}, with equality in a family of examples. There are also effective bounds known when $H(A)$ is a $d$-simplex, as we will discuss in the next subsection.

Suppose that $x$ belongs to   the right-hand side of \eqref{eq extremal union}.
To prove Theorem~\ref{thm main theorem} when $x$ is far away from the boundary of $NH(A)$ we develop an effective version of Proposition 1 of Khovanskii's original paper \cite{Kh92} using  quantitative linear algebra. Otherwise $x$ is close to a separating hyperplane of $NH(A)$: Suppose
the hyperplane is $z_d=0$; write each $a=(a_1,\dots,a_d)$ and $x=(x_1,\dots,x_d)$, so that every $a_d\geqslant 0$ and $x_d$ is ``small''. Now $x = \sum_{a\in A} m_a a$ where each $m_a \in \mathbb{Z}_{\geqslant 0}$
as  $x \in \mathcal{P}(A)$   and so $\sum_{a \in A, a_d\ne 0} m_aa_d \leqslant x_d$  is small. The contribution from those $a$ with $a_d=0$ is a ``smaller dimensional problem'', living in the hyperplane $z_d = 0$. Carefully formulated, one can apply induction on the dimension to bound $\sum_{a \in A} m_a$, and hence show that $x \in NA$. 

The structure \eqref{eq extremal union} is evidently related to Khovanskii's theorem.  However, we have not been able to find a precise way to relate Khovanskii's theorem and Theorem \ref{thm main theorem}. Our proofs of Theorem \ref{thm:effectiveKhovanksii} and Theorem \ref{thm main theorem} are almost entirely disjoint, and we get a different quality of bound in each theorem. \\

\subsection*{The size and structure of $NA$ when $H(A)$ is a $d$-simplex} 

If $A \subset \mathbb{R}^d$ then the convex hull $H(A)$ is a\emph{ $d$-simplex} if there exists
  $B \subset A$ with $\vert B\vert = d+1$ for which $B-B$ spans $\mathbb{R}^d$ and $H(A) = H(B)$ (whence $\ex(H(A)) = B$).

  \begin{Theorem}[Effective Khovanskii, simplex case]
\label{thm: effectuve Khovanksii simplex case}
If $A \subset \mathbb{Z}^d$ is finite and  $H(A)$ is a $d$-simplex then  $\vert N A\vert = P_A(N)$
 for all $N \geqslant 1$ for which
 \begin{equation}
\label{eq:effectivebound}
 N \geqslant  (d+1)! \vol(H(A)) - (d+1)(\vert A\vert -d )-d+1.
 \end{equation}
 \end{Theorem}
  
\begin{Theorem}[Effective structure, simplex case]
\label{thm: structure simplex case}
If $A \subset \mathbb{Z}^d$ is finite and  $H(A)$ is a $d$-simplex then   
  \eqref{eq extremal union} holds for all $N \geqslant 1$ for which
  \begin{equation} \label{eq:effectivebound2}
 N \geqslant  (d+1)! \vol(H(A)) - (d+1)(\vert A\vert -d ),
 \end{equation}
  and 
  if $|A|=d+1$ or $d+2$ then    \eqref{eq extremal union} holds for all  $N\geqslant 1$.
\end{Theorem}

Therefore if $A \subset \mathbb{Z}^d$ is finite and $H(A)$ is a $d$-simplex then \begin{equation}
\label{simplex Khovanksii bound}
N_{\text{Kh}}(A) \leqslant    (d+1)! \vol(H(A)) - (d+1)(\vert A\vert -d )-d+1
\end{equation} and
\begin{equation}
\label{simplex structure bound}
N_{\text{Str}}(A)\leqslant    (d+1)! \vol(H(A)) - (d+1)(\vert A\vert -d ).
\end{equation}

\noindent The hypotheses imply that $|A|\geqslant d+1$. If $d=1$ our bound gives $N_{\text{Str}}(A)\leqslant 2w(A)-2\vert A\vert +3$ which is weaker than the bound 
$N_{\text{Str}}(A)\leqslant w(A)-\vert A\vert +2$ from \cite{GW20}, which suggests that Theorem \ref{thm: structure simplex case} is still some way from being ``best possible''. 

Even though the main bounds in Theorems \ref{thm: effectuve Khovanksii simplex case} and \ref{thm: structure simplex case} are very similar, we have not been able to find a way to directly deduce one theorem from the other. Instead, we present separate arguments for each theorem (in Sections \ref{3} and \ref{4} respectively), albeit based on the same fundamental lemmas in Section \ref{2}. \\

Curran and Goldmakher \cite{CG20} gave similar (but slightly weaker) bounds in the simplex case. In \cite[Theorem 1.4]{CG20} they showed that $N_{\Kh}(A) \leqslant (d+1)! \vol(H(A)) - 3d -1$, and in \cite[Theorem 1.3]{CG20} they showed that $N_{\Str}(A)\leqslant (d+1)! \vol(H(A))  - 2d - 2$. (In the statement of \cite[Theorem 1.3]{CG20} they replace \eqref{eq extremal union} by 
\[
 NA = \bigcap\limits_{b \in \ex(H(A))}(bN + \mathcal{P}(A-b)) 
\]  but these expressions are equivalent.) Our bounds \eqref{simplex Khovanksii bound} and \eqref{simplex structure bound} match these expressions when $\vert A\vert = d+2$, but are an improvement as soon as $\vert A\vert \geqslant d+3$. 

The proofs of Theorems \ref{thm: effectuve Khovanksii simplex case} and \ref{thm: structure simplex case} look seemingly very different from the work in \cite{CG20}. Our method manipulates $A$ directly using additive-combinatorial language; Curran and Goldmakher, being inspired by Ehrhart theory, used generating functions such as $S(t) := \sum_{N \geqslant 0}\vert NA\vert t^N$), and `raised the dimension' by examining further properties of subsets of $\mathbb{Z}^{d+1}$ generated by $\{(a,1): \, a \in A\}$. 

However, the two approaches are in fact closely related. The central notion of our method for the simplex case is that of a `$B$-minimal element', see Definition \ref{Definition B minimal} below; this is equivalent to the notion of `minimal elements' defined in \cite{CG20}, at the end of page 7 and in the remark following the statement of Proposition 4.1 of that paper.  There are also analogies between some of our preparatory lemmas and partial results in \cite{CG20}, which will be discussed in Sections 3, 4, and 5 below when they occur. 

Our improvement over \cite{CG20} comes from refining an additive combinatorial lemma concerning the $B$-minimal elements, related to the Davenport constant of the group $\mathbb{Z}^d/\Lambda_{B-B}$. The key results are Lemmas \ref{lem:savingofH} and \ref{Lemma: simplex davenport style argument} below. In fact, it would have been possible to derive Theorems \ref{thm: effectuve Khovanksii simplex case} and \ref{thm: structure simplex case} directly by inputting the conclusions of Lemmas \ref{lem:savingofH} and \ref{Lemma: simplex davenport style argument} into the relevant parts of the argument of \cite{CG20}, following a translation into the generating function language of \cite{CG20} (the details are discussed after Lemma \ref{Lemma: simplex davenport style argument} below). However, we think there is extra value in showing how the analysis from \cite{CG20} can be phrased -- efficiently -- in a classical additive-combinatorial language. 

Having discussed the similarities to \cite{CG20}, it should be stressed that the main work of this paper -- all parts of the proof of Theorem \ref{thm:effectiveKhovanksii}, and the technical heart of the proof of Theorem \ref{thm main theorem} -- is not related to any part of \cite{CG20}. These novel elements comprise the majority of  the present work. \\

The structure of the paper is as follows. In the next section we briefly discuss the 1-dimensional case, and in the three subsequent sections, the simplex case.  In Section~\ref{6}, we prove the effective Khovanskii theorem (Theorem~\ref{thm:effectiveKhovanksii}). In Section~\ref{5} we then prove the effective structure result (Theorem~\ref{thm main theorem}); this part may be read essentially independently of the previous section, although there is one piece of quantitative linear algebra in common. An appendix collects together some facts from the theory of convex polytopes (which are useful in Section \ref{5}).\\

\textbf{Acknowledgements}: We would like to thank the anonymous referees for their detailed analysis of the manuscript, and for making several suggestions which refined the final bounds. \\

\section{One dimension and speculations}
\label{1-d}

It might well be that for finite $A\subset \mathbb{Z}^d$
 \begin{equation}
\label{eq: conjecture}
N_{\text{Str}}(A) \leqslant  N_{\text{Kh}}(A)\leqslant d!\, \text{vol}(H(A)).
\end{equation}
We refrain from calling this speculation a conjecture, since we have not even proved it for $d=1$. However, a slight specialisation of the relation \eqref{eq: conjecture} is true when $d=1$, and we know of no counterexample for larger $d$, so it is certainly worth investigating; we make a few remarks in this section.

After translating suppose that $0 \in \ex(H(A))$. First we note that if $\mathcal{E}(b-A)=\emptyset$ for all $b \in \ex(H(A))$ then $N_{\text{Kh}}(A)=N_{\text{Str}}(A)$. Indeed, Khovanskii's theorem \cite{Kh92} and Theorem \ref{thm main theorem} imply that the Khovanskii polynomial $P_A(N)$ is equal to $\vert NH(A) \cap \Lambda_A\vert $. Since $NA \subset NH(A) \cap \Lambda_A$, we have $\vert NA\vert \leqslant P_A(N)$ for all $N$, and $\vert NA\vert = P_A(N)$ if and only if \eqref{eq extremal union} holds, and thus $N_{\Kh}(A) = N_{\Str}(A)$.

We also obtain the bounds 
$N_{\text{Kh}}(A), N_{\text{Str}}(A)<   (d+1)! \vol(H(A))$ in 
Theorems \ref{thm: effectuve Khovanksii simplex case}
 and \ref{thm: structure simplex case}, bigger than in \eqref{eq: conjecture} by a  factor of $d+1$ (and one can see where this comes from in the proof). If $d=1$ then $\text{Vol}(H(A))=w(A)$, so the inequalities
$N_{\text{Str}}(A), N_{\text{Kh}}(A)\leqslant d!\, \text{vol}(H(A))$ can be deduced from the following:

\begin{Lemma}
\label{Lemma: 1-dim}
If $A\subset \mathbb{Z}$ with $\gcd_{a\in A} a=1$ and $|A|\geqslant 3$ then $N_{\operatorname{Str}}(A), N_{\operatorname{Kh}}(A) \leqslant w(A)-1$.
 \end{Lemma}
 \begin{proof}
We may translate $A$ so that it has minimal element $0$ and largest element $b=w(A)$.
(If $|A|=2$ then $A=\{ 0,1\}$ and $N_{\text{Str}}(A)=N_{\text{Kh}}(A)=1$).  
  The main theorem of \cite{GW20} gives that $N_{\text{Str}}(A)\leqslant b-|A|+2$, which is $\leqslant w(A)-1$ for $|A|\geqslant 3$.
 
If $N\geqslant N_{\text{Str}}(A)$ then $ NA = (NH(A) \cap \mathbb{Z}^d) \setminus (  \mathcal{E}(A) \bigcup\ (bN - \mathcal{E}(b-A)))$.
Let $e_A$ denote the largest element of $\mathcal{E}(A)$, or $e_A = -1$ if $\mathcal{E}(A)$ is empty. If $bN>e_A+e_{b-A}$ then   $\mathcal{E}(A)$ and $bN - \mathcal{E}(b-A)$ are disjoint subsets of $\{0,\dots,bN\}$ so that   $\vert NA\vert = bN -c$ where $c=\vert \mathcal{E}(A)\vert +\vert \mathcal{E}(b-A)\vert-1$. Therefore
\[
N_{\text{Kh}}(A) \leqslant \max\Big\{ N_{\text{Str}}(A) , 1+ \Big\lfloor  \frac{e_A+e_{b-A}}b \Big\rfloor\Big\} .
\]

In particular if $A=\{ 0,a,b\}$ with $(a,b)=1$ then $N_{\text{Str}}(A)=1$ by \cite[Theorem 4]{GS20} and $e_A=ba-b-a$ so that $N_{\text{Kh}}(A)= \max(1,b-2)$.

Now suppose that $|A|\geqslant 4$. By \cite[Theorem 1]{D90} we have  
\[
e_A \leqslant \frac{b(b-1)}{|A|-2}-1 \text{ so that }  1+ \Big\lfloor  \frac{e_A+e_{b-A}}b \Big\rfloor < 1+ \frac{2(b-1)}{|A|-2} \leqslant b.
\]
Therefore we have $N_{\text{Kh}}(A) \leqslant b-1=w(A)-1$.
\end{proof}
Although we do not yet know whether $N_{\Str}(A) \leqslant N_{\Kh}(A)$ in general when $d=1$, the methods of Curran--Goldmakher do show something along these lines.\footnote{Michael Curran, personal communication.} For each $g \in \{0,1,\dots,b-1\}$, let $N_{\Kh,g}(A)$ denote the optimal threshold for which $\vert NA \cap \{n: n \equiv g \, \text{mod} \, b\}\vert = P_g(N)$ for all $N \geqslant N_{\Kh,g}(A)$, where $P_g$ is some fixed polynomial; let $N_{\Str,g}(A)$ denote the optimal threshold for which \[ NA \cap \{n: n \equiv g \, \text{mod} \, b\} = ([0,bN] \cap \{n: n \equiv g \, \text{mod} \, b\}) \setminus (\mathcal{E}(A) \cup (bN - \mathcal{E}(b-A)))\] for all $N \geqslant N_{\Str,g}(A)$. Then \begin{equation}
\label{eq:conjecturespecial}
 N_{\Str,g}(A) \leqslant N_{\Kh,g}(A).
 \end{equation} This is obtained by considering the proofs in \cite[Section 3]{CG20}, which show that $N_{\Kh,g}(A) = \deg P - d$ when $H(A)$ is a simplex, where $P$ is some auxiliary polynomial: In \cite[Section 4]{CG20} Curran--Goldmakher then show that $N_{\Str,g}(A) \leqslant \deg P - 1$ for the same auxiliary polynomial $P$. Unfortunately, although $N_{\Str}(A) = \max_g N_{\Str,g}(A)$, one could potentially get $N_{\Kh}(A) < \max_g N_{\Kh,g}(A)$, so the inequality \eqref{eq:conjecturespecial} does not immediately give \eqref{eq: conjecture} when $d=1$.

Curran--Goldmakher also give the precise value of $N_{\text{Kh}}(A)$ in \eqref{eq CGbound} in certain special cases including the useful example $A: = \{(0,\dots,0), (1,\dots,1), m_1 e_1,\dots,m_de_d\} \subset \mathbb{Z}^d$ where the $m_j$ are pairwise coprime positive integers and  the $e_1,\dots,e_d$ are the standard basis vectors. If all the $m_j$ are close to $x$ so that $w(A)\approx x$ for some large $x$ then $N_{\operatorname{Kh}}(A)\sim_{x\to \infty} w(A)^d$, which suggests we might be able to reduce the bound in Theorem \ref{thm:effectiveKhovanksii} to $w(A)^d$. However  $d!\, \text{vol}(H(A))$ would be a preferable bound to  $w(A)^d$, since it is smaller and more precise in the example where we let $m_2=\dots=m_d=1$ and $m_1=x$ be arbitrarily large
so that $N_{\operatorname{Kh}}(A)\sim_{x\to \infty} w(A)$.\\

\section{Preparatory lemmas for the simplex case}
\label{2}


 Throughout this section,   $0 \in A \subset \mathbb{Z}^d$ and $A$ is finite. Let $N_A(0) = 0$ and for each $v \in \mathcal{P}(A) \setminus \{0\}$   let $N_A(v)$ denote the minimal positive integer $N$ such that $v \in NA$.

\begin{Definition}[$B$-minimal elements]
Suppose that  $B \cup \{0\} \subset A \subset \mathbb{Z}^d$, with $A$ finite. Let $\mathcal{S}(A,B)$ denote the set of \emph{$B$-minimal elements}\footnote{We observe that 
 $\mathcal{S}(A,B)$ is the set of  $u \in \mathcal{P}(A)$ such that $(u,N_A(u)) \in \mathbb{Z}^{d+1}$ is minimal  in the sense of \cite[Section 3,4]{CG20}, in particular the bottom of page 7 and the remark following Proposition 4.1 of that paper.},  which comprises $0$ and those elements  $u \in \mathcal{P}(A) \setminus \{0\}$ such that $a_i\not\in B\cup \{0\}$ for every $i$ whenever
 \[ u = a_1 + a_2 + \cdots + a_{N_A(u)} \text{ with each } a_i \in A. \]
 \end{Definition}

$B$-minimal elements can be used to decompose $NA$ and $\mathcal{P}(A)$ into simpler parts. The following is the analogous statement to \cite[Proposition 4.1]{CG20}, although that proposition is only stated in the case when $H(A)$ is a $d$-simplex. 

\begin{Lemma}
\label{Lemma: B minimal sumset}
If $B^*:=B \cup \{0\} \subset A \subset \mathbb{Z}^d$ with $A$ finite then
 \[
 \mathcal{P}(A) = \mathcal{S}(A,B) + \mathcal{P}(B^*) \text{ and } 
 NA = \bigcup\limits_{\substack{u \in \mathcal{S}(A,B)\\ N_A(u) \leqslant N}} (u + (N-N_A(u))B^*).
 \]
\end{Lemma}
\begin{proof}
The second assertion implies the first by taking a union over all $N$. That each 
$u + (N-N_A(u))B^* \subset NA$ is immediate, so we need only show that if $v\in  NA$ then $v\in u + (N-N_A(u))B^*$ for some 
$u\in  \mathcal{S}(A,B)$ with $N_A(u) \leqslant N$.

Now, for any $v\in  NA$  we can write 
\[
v = u+w \text{ with } u=a_1 + \cdots + a_L \text{ and } w=b_1 + \cdots + b_M,
\] 
where $L,M \geqslant 0$, and each $a_i \in A \setminus B$ and $b_i \in B$, with $M$ maximal and  $L +M = N_A(v)$.
Then $N_A(u) = L$ and $N_A(w) = M$, else we could replace the above expression for $u$ or $w$ by a shorter sum of elements of $A$, and therefore obtain a shorter sum of elements to give $v$, contradicting that $L+M=N_A(v)$ is minimal. Moreover $u\in S(A,B)$ else we could replace the sum $a_1 + \cdots + a_L$ in the expression for $v$ by a different sum of length $L$ which includes some elements of $B$,  contradicting the maximality of $M$.
 
Therefore $u \in S(A,B)$ with $N_A(u)=L\leqslant  N_A(v)\leqslant N$ and
\[
v \in u + MB^* = u + (N_A(v) - N_A(u))B^* \subset u + (N - N_A(u))B^*,
\]  
since $0\in B^*$.
\end{proof}

It will be useful to control the complexity of the $B$-minimal elements. 
\begin{Definition}
\label{Definition B minimal}
Let $B \cup \{0\} \subset A \subset \mathbb{Z}^d$, with $A$ finite. If $\mathcal{S}(A,B)$ is a finite set, we define \[K(A,B) := \max \limits_{u \in \mathcal{S}(A,B)} N_A(u).\]
\end{Definition}
\noindent In certain circumstances we will   bound $K(A,B)$ using results on  \emph{Davenport's problem}, which asks for the smallest integer $D(G)$ such that any set of $D(G)$ (not necessarily distinct) elements of an abelian group $G$ contains a subsum\footnote{A \emph{subsum} of a given   sum $\sum_{i \in I} g_i$   is a sum of the form $\sum_{i \in I^\prime} g_i$ where $I^\prime \subset I$ is  non-empty, of  \emph{length} $\vert I^\prime \vert$.}  that equals $0_G$. It is known that  $D(G)\leqslant m(1+\log |G|/m)$ where  $m$ is the maximal order of an element of $G$.

\begin{Definition}
Given a finite abelian group $G$, if $0\not\in H \subset G$ let  $k(G,H)$ be the length of the longest sum of elements of $H$ which contains no subsum equal to $0$,
and  no subsum of length $>1$ belonging to $H$. 
\end{Definition}

\begin{Lemma}
\label{lem:savingofH}
Given a finite abelian group $G$, for any $0 \notin H \subset G$ we have $k(G,H) \leqslant \vert G\vert - \vert H \vert$. Moreover 
$k(G,H)\leqslant  m(1+\log |G|/m)-1$, where $m$ is the maximal order of an element of $G$. 
\end{Lemma}

\begin{proof}
Suppose we are given a longest sum $h_1+\dots+h_k$ of elements of $H$ defining $k(G,H)$, so that   $k = k(G,H)$. Then
\[
0, h_1+h_2, h_1+h_2+h_3, \dots, h_1+\dots+h_k,  
\]
 are all distinct in $G$, else subtracting would give a subsum equal to $0$, and they are all contained in $G \setminus H$. Therefore  $k+|H| \leqslant |G|$ and the first result follows.
 
 By definition $k(G,H)<D(G)$ so the second result claims from the result noted for $D(G)$ above.
  \end{proof}
  
Curran and Goldmakher's  \cite[Lemma 3.1]{CG20}    implies   the weaker  upper bound   $k(G,H)\leqslant \vert G\vert - 1$. This difference leads in part to the improvements in Theorems \ref{thm: effectuve Khovanksii simplex case} and  \ref{thm: structure simplex case}.

\subsection{$d$-dimensional simplices}
Let $B=\{ b_1,\dots,b_d\}\subset A$ be a basis for $\mathbb{R}^d$ with
\[
B \cup \{0\} \subset A \subset H(B \cup \{0\}) \text{ and } A \subset \mathbb{Z}^d
\]
so that $A$ is finite.
Since $C_A = C_B$, and $B$ is a basis, there is a unique representation of every vector $r\in C_A$ as 
\begin{equation} \label{eq: r-rep}
r = \sum_{i=1}^d r_ib_i \text{  where each } r_i\geqslant 0.
\end{equation} 
If $r\in H(A) = H(B \cup \{0\})$ then  $\sum_{i=1}^d r_i\leqslant 1$.

\begin{Lemma}
\label{Lemma: simplex basic davenport argument}
Suppose  $B=\{ b_1,\dots,b_d\}$ is a basis for $\mathbb{R}^d$ with $B \cup \{0\} \subset A \subset H(B \cup \{0\}) $ and 
$A \subset \mathbb{Z}^d$ is finite.
If $r\in \mathcal P(A)$ and  $r \equiv a  \pmod {\Lambda_B}$ with  $a\in A$  then
$r-a\in \mathcal P(B \cup \{0\})$ (where we choose $a=0$ if $r\in \Lambda_B$).
\end{Lemma}

 \begin{proof} Since $r\in \mathcal P(A)\subset C_A$, we have the representation \eqref{eq: r-rep} for $r$. 
 Moreover since $a\in H(A)=H(B \cup \{0\})$  we have the representation $a = \sum_{i=1}^d a_i b_i$ by \eqref{eq: r-rep} with 
 $\sum_{i=1}^d a_i \leqslant 1$.  If   $a\not\equiv 0 \pmod {\Lambda_B}$ then each $a_i\in [0,1)$, and otherwise we choose 
  $a=0$ so each $a_i=0$.  Therefore $\sum_{i=1}^d r_i b_i=r\equiv a=\sum_{i=1}^d a_ib_i \pmod {\Lambda_B}$, and each $r_i\equiv a_i \pmod 1$.
 As each $r_i \geqslant 0$ we write $m_i=r_i - a_i$ so that each $m_i\in \mathbb Z_{\geqslant 0}$ and $r-a=\sum_{i=1}^d m_ib_i  \in \mathcal P(B \cup \{0\})$.
\end{proof}

We   use this lemma to  bound $K(A,B)$. 

\begin{Lemma}
\label{Lemma: simplex davenport style argument}
Suppose  $B=\{ b_1,\dots,b_d\}$ is a basis for $\mathbb{R}^d$ with $B \cup \{0\} \subset A \subset H(B \cup \{0\}) $ and 
$A \subset \mathbb{Z}^d$ is finite.
 If   $u = a_1 + \cdots + a_{N_A(u)}\in \mathcal S(A,B)$ is non-zero then any subsum with two or more elements cannot belong to $A_B:=A \text{ mod }\Lambda_B$, and no subsum can be congruent to $0 \text{ mod }\Lambda_B$. Therefore
\[
K(A,B) \leqslant k( \Lambda_A / \Lambda_B, A_B \setminus   \{0\} )  .
\] 
\end{Lemma}

\begin{proof}
Let  $r$ be a subsum of $a_1 + \cdots + a_{N_A(u)}$ of size $\ell> 1$. Then $\ell = N_A(r)$ and $r\in  \mathcal S(A,B)$ as $u\in  \mathcal S(A,B)$.
We write $r$ as in \eqref{eq: r-rep} so that  $\sum_{i\leqslant d} r_i\leqslant \ell = N_A(r)$.
Suppose that $r \equiv a  \pmod {\Lambda_B}$ for some  $a\in A$  
(where we choose $a=0$ if $r\in \Lambda_B$) so that  $m:=r-a\in \mathcal P(B \cup \{0\})$  by  Lemma \ref{Lemma: simplex basic davenport argument}.
Therefore $N_A(m)\geqslant \ell-N_A(a) \geqslant \ell-1>0$ (so $m\ne 0$).   On the other hand  $N_A(m)=\sum_{i\leqslant d} (r_i-a_i) = \ell$ if $a=0$, and
$<\ell$ if $a\ne 0$, so $N_A(m)\leqslant \ell-N_A(a)$. We deduce that $r$ can be represented as $a$ plus the sum of $\ell-N_A(a)$ elements of $B$, contradicting that $r\in \mathcal S(A,B)$.
\end{proof}
The combination of Lemmas \ref{lem:savingofH}, \ref{Lemma: simplex basic davenport argument} and  \ref{Lemma: simplex davenport style argument} effects an upper bound bound on $N_A(u)$ when $u \in \mathcal{S}(A,B)$, which is analogous to the bound from the statement of \cite[Lemma 3.1]{CG20} (albeit slightly stronger due to the stronger bound on $k(G,H)$ in this paper). \\

If the convex hull of $A$ is not a simplex 
then $\mathcal{S}(A,B)$ need not  be finite. For example,  if  $B = \{(0,1),(1,0)\}\subset A = \{(0,0), (-1,1), (0,1), (1,0)\}$ then $\mathcal{S}(A,B) = \{(-k,k): k \in \mathbb{Z}_{ \geqslant 0}\}$. This is one   reason why $\mathcal{S}(A,B)$ is not used later in Section \ref{5}, when dealing with general sets $A$. 

\subsection{Translations}
We finish by observing that under rather general hypotheses the sets $\mathcal{S}(A,B)$, and consequently the quantities $K(A,B)$, are well-behaved under translations. This observation was also made in \cite[Lemma 4.2]{CG20}. 

\begin{Lemma}
\label{Lemma Sets S(A,B)}
Let $B \cup \{0\} \subset A \subset \mathbb{Z}^d$, with $A$ finite. If $b\in B$ then
\[
\mathcal{S}(b-A,b-B) = \{ bN_A(u)-u:\ u\in \mathcal{S}(A,B)\}  
\]
and if $v= bN_A(u)-u$ with $u \in \mathcal{S}(A,B)$ then $N_{b-A}(v)=N_A(u)$. In particular we have $K(b-A,b-B)=K(A,B)$.
 \end{Lemma}
\begin{proof} Let $N=N_A(u)$.
If $u = a_1 + a_2 + \cdots + a_{N}$ then $v:=bN-u=(b-a_1)+\dots +(b-a_{N})$ so that $N_{b-A}(v)\leqslant N_A(u)$.
If $N_{b-A}(v)\leqslant N-1$ say $v=(b-a_1')+\dots +(b-a_M')$ with $M<N$ then $u=a_1'+\dots +a_M'+ (N-M)b$ contradicting the definition of $u \in\mathcal{S}(A,B)$. We deduce that there is a 1-to-1 correspondance between the representations of $u$ as the sum of $N$ elements of $A$, and of $v$ as the sum of $N$ elements of $b-A$, and the result follows.
\end{proof}

\section{Structure bounds in the simplex case}\label{3}

First we deal with the special cases. 

\begin{proof}[Proof of Theorem \ref{thm: structure simplex case} for $|A|=d+1$ and $|A|= d+2$]  
Let $\ex(H(A)) = B$ where $\vert B\vert = d+1$ and $\spn(B-B) = \mathbb{R}^d$. Write $B = \{b_0,\dots,b_{d}\}$, and translate so that $0 = b_0 \in B$. 

If $|A|=d+1$ then $B=A$ and $\mathcal{E}(b-A) = \emptyset$ for all $b \in B$. We immediately see that $NA = NH(B) \cap \Lambda_{B}= NH(A)\cap  \Lambda_A$ for all $N\geqslant 1$.

If $|A|= d+2$ write  $A=B\cup \{ a\}$. Since $a\in H(B)$ we can write
$a=\sum_{i=0}^d a_i b_i$ uniquely with each $a_i\geqslant 0$ and $\sum_{i=0}^d a_i=1$. We know that the finite group $\Lambda_A/\Lambda_B$  is generated by $a$. If $a$ has order $M$ in the group $\Lambda_A/\Lambda_B$ then
 the classes of $\Lambda_A/\Lambda_B$ can   be represented by 
\[ 
\mathcal{S}(A,B)=\{ ma: 0\leqslant m\leqslant M-1\} .
\] 

Now let \begin{equation}
\label{eq 1 where v is}
v \in (NH(A) \cap \Lambda_A) \setminus (\bigcup\limits_{b \in B} (bN - \mathcal{E}(b-A))).
\end{equation} Since $v \in NH(A)=  NH(B)$ we can write $v=\sum_{i=0}^d v_i b_i$ in a unique way with each $v_i\geqslant 0$ and $\sum_{i=0}^d v_i =N$. This implies that $b_iN-v = \sum_{j=0}^dv_j(b_i - b_j) \in C_{b_i - A}$, and from \eqref{eq 1 where v is} we have $b_iN - v \in \Lambda_{A} = \Lambda_{b_i - A}$ and $b_iN - v \notin \mathcal{E}(b_i - A)$. Hence $b_iN - v \in  \mathcal P(b_i-A)$ for all $i$, in particular $v \in \mathcal{P}(A)$ (from $i=0$).

Suppose that $v\equiv ma  \mod \Lambda_B$ for some $0 \leqslant m \leqslant M-1$. This implies that $v_i - ma_i \in \mathbb{Z}$ for  $i=0,1,\dots,d$, and we now show that $v_i - ma_i \in \mathbb{Z}_{ \geqslant 0}$ if $i\ne 0$:  Since $v \in \mathcal{P}(A)$ we may write 
\[
v = (m + \lambda M)a + \sum_{i=1}^d (v_i-(m + \lambda M)a_i)b_i
\] 
for some $\lambda \in \mathbb{Z}_{ \geqslant 0}$ with   $v_i - (m + \lambda M)a_i \in \mathbb{Z}_{ \geqslant 0}$ for $i=1,\dots,d$. Therefore we conclude that $v_i - ma_i \in \mathbb{Z}_{ \geqslant 0}$ for all $i \geqslant 1$.

We now give an analogous argument for representations of $b_jN-v$ for each $j=1,\dots,d$: For each $j$ we also have
\[ 
b_jN-v =   \sum_{i=0}^d v_i   (b_j-b_i)   \equiv  \sum_{i=0}^d ma_i (b_j-b_i) = m(b_j-a) \text{ mod } \Lambda_{b_j - B}
\]  
Since $b_jN-v\in \mathcal P(b_j-A)$ we may write
\[
b_jN-v = (m + \lambda M)(b_j-a)  + \sum_{i=0}^d (v_i-(m + \lambda M)a_i)(b_j-b_i)  
\]
for some $\lambda \in \mathbb{Z}_{ \geqslant 0}$ with $v_i - (m + \lambda M)a_i \in \mathbb{Z}_{ \geqslant 0}$  for $i=0,\dots,d$ with  $i\ne j$  (we can't deduce this for $i=j$ since then $b_j-b_i=0$).  Therefore
 $v_i\geqslant ma_i$ for all $i\ne j$.  

Combining these observations, we deduce that $v_i-ma_i\in \mathbb Z_{\geqslant 0}$ for all $i$, which implies
  that 
\[
v = ma + \sum_{i=0}^d (v_i-ma_i)b_i \in \bigg(m+\sum_{i=0}^d (v_i-ma_i) \bigg) A=NA. \qedhere
 \]
 \end{proof}

We now prove the rest of Theorem \ref{thm: structure simplex case}. We'll use our bound on $K(A,B)$ from Lemma \ref{Lemma: simplex davenport style argument}, combined with the following theorem. 

\begin{Theorem}
\label{thm: structure theorem simplex K}
Let $A \subset \mathbb{Z}^d$ be a finite set, for which  $H(A)$ is a $d$-simplex and $0 \in B:=\ex(H(A))$.
Then \eqref{eq extremal union} holds for all $N\geqslant (d+1)(K(A,B)-1)$.
 \end{Theorem}
 This result can be abstracted from the proof of \cite[Lemma 3.2]{CG20} and the part of the proof of \cite[Theorem 1.3]{CG20} following expression (11). 
  
\begin{proof}  
The proof follows similar lines to \cite{GS20}. For all
\begin{equation}
\label{eq where v is}
v \in (NH(A) \cap (a_0N + \Lambda_{A-A})) \setminus (\bigcup\limits_{b \in B} (bN - \mathcal{E}(b-A)))
\end{equation}
we wish to show that $v \in NA$. Now  $v \in NH(A) =  NH(B)$, so if $B = \{0=b_0,b_1,\dots,b_d\}$ then $v = \sum_{i=0}^d v_i b_i$ for some $v_i \in \mathbb{R}_{\geqslant 0}$ with $\sum_{i=0}^d v_i = N$. We will now show that 
$v\in N_jA$ for each $j$, where $N_j=K(A,B)+\sum_{i\ne j}  \lfloor v_i \rfloor$:

Taking $j=d$ (all other cases are analogous), we observe that 
\[ 
b_dN - v = \sum\limits_{i=0}^{d-1}v_i(b_d - b_i) \in (N-v_d)\cdot H(b_d - B) ,
\] 
so that  $b_dN - v \in C_{b_d - B}=C_{b_d - A}$. Therefore $b_dN - v \in  \mathcal{P}(b_d - A)$, as $b_dN - v \notin \mathcal{E}(b_d - A)$ and $b_dN - v \in \Lambda_{b_d - A}$ by \eqref{eq where v is}. So we may write 
 \[
 b_dN - v = u + w \text{ with }u \in \mathcal{S}(b_d - A, b_d - B) \text{ and } w \in \mathcal{P}(b_d - B)
 \] 
 by Lemma \ref{Lemma: B minimal sumset}.  Then $w=\sum_{i=0}^{d-1}w_i(b_d - b_i)$ for some $w_i\in \mathbb Z_{\geqslant 0}$, which implies $0\leqslant w_i\leqslant v_i$  so that $w_i\leqslant \lfloor v_i \rfloor$ for each $i$. But then $w\in (\sum_{i\ne d}  \lfloor v_i \rfloor)B\subset (N_d-K(A,B))A$ and $u\in K(A,B)A$ since $K(A,B)=K(b_d - A, b_d - B)$   by Lemma \ref{Lemma Sets S(A,B)}. Therefore $v=u+w\in N_dA$ as claimed.

We have $v \in NA$ if $\sum_{i\ne j}  \lfloor v_i \rfloor\leqslant N-K$ for some $j$, where $K=K(A,B)$. If not then 
$\sum_{i\ne j} v_i\geqslant \sum_{i\ne j}  \lfloor v_i \rfloor\geqslant N-K+1$ for each $j$, and so
\[
N= \sum_{i=0}^d v_i = \frac 1d   \sum_{j=0}^d\sum_{i\ne j} v_i \geqslant \frac{d+1}d (N-K+1),
\]
which would imply that $N\leqslant (d+1)(K-1)$. Therefore $v \in NA$ when $N > (d+1)(K-1)$. 

If $N=(d+1)(K-1)$ and the above inequalities fail to yield a contradiction, the last two chains of inequalities must be equalities. Therefore each $v_i\in \mathbb Z$, and so $u=0$, (since $0$ is the only element in $\mathcal{S}(b_d - A, b_d - B)$ that is congruent to $0$ mod $\Lambda_{b_d - B}$. This implies that $v=w\in (N_d-K(A,B)) A = (\sum_{i \neq d} v_i) A \subset NA$ as required.  
 \end{proof}

\begin{proof}[Proof of Theorem \ref{thm: structure simplex case} for $|A|\geqslant d+3$]   
Now  $A\setminus B$ is non-empty. Replacing $A$ with $A-b$ (for some $b \in \ex(H(A))$ we may assume, without loss of generality, that $0 \in \ex(H(A)) = B$. Applying Lemma \ref{Lemma: simplex davenport style argument} and Lemma \ref{lem:savingofH}, we then have
\begin{align*}
K(A,B) & \leqslant k(\Lambda_A /\Lambda_B, A_B \setminus \{0\}) \\
& \leqslant  \vert \Lambda_A/\Lambda_B \vert - \vert A\vert +  \vert B\vert \\ 
& \leqslant \vert \mathbb{Z}^d /\Lambda_B \vert - \vert A\vert +  \vert B\vert  
 = d! \vol(H(A)) - \vert A\vert + d + 1.
\end{align*} 
By Theorem \ref{thm: structure theorem simplex K}, we conclude that  \eqref{eq extremal union} holds for all $N$ in the range  \eqref{eq:effectivebound2}, as required. The result follows.
\end{proof}

\section{The Khovanskii polynomial in the simplex case}\label{4}
\label{Section Khovanskii poly simplex case}
In this section we prove Theorem \ref{thm: effectuve Khovanksii simplex case}, and make various remarks about the form of the Khovanskii polynomial itself. By analogy with the previous section, the main technical result is as follows:

\begin{Theorem}
\label{thm: khovanskii theorem simplex K}
Let $A \subset \mathbb{Z}^d$ be a finite set, for which  $H(A)$ is a $d$-simplex and $0 \in B:=\ex(H(A))$.
Then $\vert NA\vert = P_A(N)$ for all $N \geqslant 1$ for which $N\geqslant (d+1)K(A,B) - 2d$. 
 \end{Theorem}
 \noindent This same result may be extracted from the proofs of \cite[Lemma 3.2]{CG20} and \cite[Theorem 1.4]{CG20} on pages 9 and 10 of that paper. 
 \begin{proof}
We write $B=\{ 0,b_1,\dots,b_d\}$ where $\{b_1,\dots,b_d\}$ is a basis for $\mathbb{R}^d$ and
\[
B \subset A \subset H(B) \subset \mathbb{Z}^d.
\] 
For each $g\in G := \mathbb{Z}^d/\Lambda_B$ we have a coset representative $g =\sum_{i=1}^d g_i b_i$ where each $g_i\in [0,1)$. We may partition $NA$ as the (disjoint) union over $g\in G$ of \[(NA)_g:=\{ v\in NA: v\in g+\Lambda_B\},\] and thus we wish to count the number of elements in each $(NA)_g$. If
\[\mathcal{S}(A,B)_g:=\{ u\in \mathcal{S}(A,B): u\in g+\Lambda_B\}\] then, by Lemma \ref{Lemma: B minimal sumset},
 \[
 (NA)_g = \bigcup_{\substack{u\in \mathcal{S}(A,B)_g \\  N_A(u) \leqslant N}}  \bigg(  u + (N-N_A(u)) B \bigg).
 \]This union is not necessarily disjoint, but we may nonetheless develop a formula for its size by using  inclusion-exclusion.
 
It is helpful to distinguish the case when $g = 0$. In this instance $\mathcal{S}(A,B)_g = \{0\}$, and since $N_A(0) = 0$ we conclude that for all $N \geqslant 1$, \[ \vert (NA)_0\vert = \vert NB\vert = \# \{ (\ell_1,\dots,\ell_d)\in \mathbb Z_{\geqslant 0}^d: \ell_1+\cdots+\ell_d\leqslant N\}=\binom{N+d}{d},\] and this is a polynomial in $N$, namely $\frac{1}{d!} (N+d) \cdots (N+1)$. 

Now we consider the case $g \neq 0$. Let $\mathcal{S}(A,B)_g=\{ u_1,\dots,u_k\}$, as  $\mathcal{S}(A,B)$ is finite by Lemma \ref{Lemma: simplex davenport style argument}, and so 
write 
\[
u_j=g+\sum_{i=1}^d u_{j,i} b_i = a_1+\cdots + a_{N_A(u_j)}
\]
where each $u_{j,i}\in \mathbb Z_{\geqslant 0}$.
Expressing each $a_{\ell}$ in terms of the basis $\{b_1,\dots,b_d\}$, and using the fact that $g \neq 0$, we deduce that 
\[
 \sum\limits_{i=1}^d u_{j,i} < N_A(u_j) \text{ so that } \Delta_j:=N_A(u_j)-\sum_{i=1}^d u_{j,i} > 0.
 \] 
Since the $u_{j,i}$ are integers, we conclude that \[ \sum_{i=1}^d u_{j,i} \leqslant N_A(u_j) - 1.\] Therefore, if $N \geqslant N_A(u_j)$ then
\[
u_j + (N-N_A(u_j)) B  =g + \bigg\{ \sum_{i=1}^d m_i b_i:  \text{Each } m_i\in \mathbb Z_{\geqslant u_{j,i}} \text{ and } \sum_{i=1}^d m_i\leqslant N-\Delta_j\bigg\}
\] 
(and the set on the right-hand side of the above expression is empty when  $N < N_A(u_j)$). Therefore for all $N$ and for all non-empty subsets $ J \subset \{ 1,\dots,k\}$ we have
\begin{equation}
\label{eq: intersection}
\bigcap_{j\in J} \bigg( u_j + (N-N_A(u_j)) B \bigg) = g + \bigg\{ \sum_{i=1}^d m_i b_i:  \text{Each } m_i\geqslant u_{J,i} \text{ and } \sum_{i=1}^d m_i\leqslant N-\Delta_J\bigg\},
\end{equation}
 where we understand the $m_i$ to always be integers, and we let
 \[
 u_{J,i}:=\max_{j\in J} u_{j,i} \text{  and  } \Delta_J:=\max_{j\in J} \Delta_j.
 \] 
 
 Let \[
 N_J:=\Delta_J+\sum\limits_{i=1}^{d}u_{J,i}.
 \] To count the number of points in the intersection \eqref{eq: intersection} we write each $m_i=u_{J,i} +\ell_i$, and then 
 \[
 \Big\vert\bigcap_{j\in J} ( u_j + (N-N_A(u_j)) B )\Big\vert =  
\# \{ (\ell_1,\dots,\ell_d)\in \mathbb Z_{\geqslant 0}^d: \ell_1+\cdots+\ell_d\leqslant N-N_J\}=\binom{N-N_J+d}{d},
 \]
 where we define $\binom{N-N_J+d}{d}:=0$ if $N<N_J$. Hence, by inclusion-exclusion we obtain
 \begin{align*}
 \# (NA)_g &= \sum_{\substack{ J\subset \{ 1,\dots,k\}\\ |J|\geqslant 1 }} (-1)^{|J|-1}  \Big\vert\bigcap_{j\in J} ( u_j + (N-N_A(u_j)) B )\Big\vert  \\ 
 &= \sum_{\substack{ J\subset \{ 1,\dots,k\}\\ |J|\geqslant 1 }} (-1)^{|J|-1} \binom{N-N_J+d}{d} .
 \end{align*}
In fact this formula extends to cover the case $g=0$, taking $k=1$ and $N_{\{1\}} := 0$. Therefore we have the general formula
 \begin{equation}
\label{eq: General.Formula}
  \# NA = \sum_{g\in G}  \# (NA)_g =  \sum_{\substack{g\in G \\ \mathcal{S}(A,B)_g=\{ u_1,\dots,u_k\} }} \sum_{\substack{ J\subset \{ 1,\dots,k\}\\ |J|\geqslant 1 }} (-1)^{|J|-1} \binom{N-N_J+d}{d} .
\end{equation}

We wish to replace the binomial coefficients in this formula by polynomials in $N$; that is,
\[
\text{Replacing } \binom{N - N_J + d}d  \text{ by }  \frac{1}{d!} (N - N_J + d) \cdots (N - N_J + 1),
\]
but these are only equal if $N\geqslant N_J-d$. Therefore we are guaranteed that 
\[
 \# (NA)_g =  P_g(N) \text{ where } P_g(T) :=  \sum_{\substack{ J\subset \{ 1,\dots,k\}\\ |J|\geqslant 1  }} (-1)^{|J|-1} \frac{(T-N_J+d)\cdots (T-N_J+1)} {d!},
\]
provided $N\geqslant \max_J N_J-d=N_{\{ 1,\dots,k\} }-d$.  Therefore
\[
 \# NA = P_A(N) \text{ where } P_A(T):= \sum_{g\in G} P_g(T)
 \]
 once $N \geqslant 1$ (the trivial bound from the $g=0$ class) and $N  \geqslant \max_{g \neq 0} (N_{\{1,\dots,k\}} - d)$. 
 
 It remains to bound $N_{\{1,\dots,k\}}$. By definition we have
\begin{equation}
\label{eq: simple N bound}
 N_{\{1,\dots,k\}} \leqslant \max_j N_A(u_j) + \sum\limits_{i=1}^d \max_j u_{j,i} \leqslant K(A,B) + \sum\limits_{i=1}^d \max_j u_{j,i}.
 \end{equation} 
Now \[u_{j,i} \leqslant \sum_{i=1}^d u_{j,i} \leqslant N_A(u_j) - 1 \leqslant K(A,B) - 1\] by definition. Thus \[ N_{\{1,\dots,k\}}-d \leqslant (d+1) K(A,B)- 2d \] as claimed. \end{proof}

We remark that the $-2d$ term (in the $(d+1)!K(A,B) - 2d$ bound from Theorem \ref{thm: khovanskii theorem simplex K}) was saved by two separate actions. First, $-d$ was saved through considering $g=0$ and $g \neq 0$ separately; there is an equivalent manoeuvre on \cite[Page 9]{CG20} when it is assumed that `$\boldsymbol{g_i}$ is not congruent to $\boldsymbol{0}$'. Then, $-d$ was saved by noting that the binomial coefficient $(\begin{smallmatrix} N - N_J + d \\ d \end{smallmatrix})$ agrees with the polynomial $\frac{1}{d!} (N - N_J + d) \cdots (N - N_J + 1)$ for all $N \geqslant N_J - d$ not just for all $N \geqslant N_J$. This is analogue to the $-d$ that is saved by the application of the division algorithm in \cite[Proof of Theorem 1.4]{CG20} at the bottom of page 10 of that paper. \\ 
\begin{proof}[Proof of Theorem \ref{thm: effectuve Khovanksii simplex case}] As in the proof of Theorem \ref{thm: structure simplex case} at the end of Section \ref{simplex structure bound}, we may replace $A$ with $A-b$ (for some $b \in \ex(H(A))$ and assume without loss of generality that $0 \in \ex(H(A)) = B$. We again have the bound \[ K(A,B) \leqslant d! \vol(H(A)) - \vert A\vert + d + 1,\] which substituting into Theorem \ref{thm: khovanskii theorem simplex K} shows that $\vert NA\vert = P_A(N)$ in the range required.
\end{proof}

 \subsection{Smaller $N$} \label{sec: small}

Returning to the proof of Theorem \ref{thm: khovanskii theorem simplex K}, one may sometimes show that $\# (NA)_g = P_g(N)$ for more values of $N$. 
\begin{Proposition}  Define
\[
W(h):= 
\sum_{\substack{ J\subset \{ 1,\dots,k\}\\ |J|\geqslant 1\\ N_J=N_{ \{ 1,\dots,k\} } -h}} (-1)^{|J|}, 
\]
and let $h$ be the smallest non-negative integer for which $W(h)\ne 0$. Then
$ \# (NA)_g = P_g(N) $
 for all $N\geqslant N_{\{ 1,\dots,k\} }-d-h$, but not for $N= N_{\{ 1,\dots,k\} }-d-h-1$. 
  \end{Proposition}

\begin{proof} Letting $m \geqslant 0$ and $N= N_{\{ 1,\dots,k\} }-d-1-m$ we have
 \begin{align*}
  \# (NA)_g-P_g(N)  &= \sum_{\substack{ J\subset \{ 1,\dots,k\}\\ |J|\geqslant 1\\ N_J  \geqslant  N_{\{ 1,\dots,k\} } -m   }} (-1)^{|J|} \frac{(N-N_J+d)\cdots (N-N_J+1)} {d!} \\
   &= (-1)^d \sum_{\kappa=0}^m \binom{\kappa+d}d   W(m-\kappa) 
 \end{align*}
 since if $N_J=N_{ \{ 1,\dots,k\} } -(m-\kappa)$ then 
 \[
 \frac{(N-N_J+d)\cdots (N-N_J+1)} {d!} =(-1)^d    \binom{\kappa+d}d
 \]
 If $m\leqslant h-1$ then every term on the right-hand side is $0$ and so  $ \# (NA)_g=P_g(N)$. If $m=h$ then
 $\# (NA)_g=P_g(N) + (-1)^d   W(h)$. 
  \end{proof}

 \subsection{Determing $W(0)$}

We do not see how to easily determine $h$ in general, though it is sometimes  possible to identify whether $W(0) = 0$. 

\begin{Proposition}   Let $J_0:=\{ j:  \Delta_j= \Delta_{\{1,\dots, k\}} \}$ and $J_i:=\{ j:  u_{j,i}= u_{\{1,\dots,k\},i} \}$  for $1\leqslant i\leqslant d$, with
  $J^*: = \cup_{0\leqslant i\leqslant d} J_i$.
\begin{enumerate}[(i)]
\item If $J^*$ is a proper subset of $\{1,\dots, k\}$ then $W(0)=0$.
\item If $J^* = \{1,\dots,k\}$ and, for each $i$, there exists  $j_i\in J_i$ such that $j_i\not\in J_\ell$ for any $\ell\ne i$,  then $W(0)=(-1)^{d+1}\ne 0$.
(For example, when the sets $J_i$ are disjoint.)
\end{enumerate}
 \end{Proposition}

\begin{proof} We have $N_J = N_{\{1,\dots,k\}}$ if and only if  $J \cap J_i \ne \emptyset$ for all $0\leqslant i\leqslant d$. Therefore, by inclusion-exclusion we have  
\begin{align*} W(0)& = \sum\limits_{\substack{J \subset \{1,\dots,k\} \\ \vert J \cap J_i\vert \geqslant 1 \text{ for each } i}} (-1)^{\vert J\vert} 
= \sum_{\substack{I \subset \{0,\dots,d\} \\ I \neq \emptyset} } (-1)^{|I|}
\sum\limits_{\substack{J \subset \{1,\dots,k\} \\  J \cap J_i=\emptyset \text{ for each } i\in I}} (-1)^{\vert J\vert} \\
&= \sum_{\substack{I \subset \{0,\dots,d\} \\ I \neq \emptyset}} (-1)^{|I|}
\sum\limits_{\substack{J \subset \{1,\dots,k\} \setminus \cup_{i\in I}  J_i }} (-1)^{\vert J\vert}
= \sum_{\substack{ I \subset \{0,\dots,d\} \\  \cup_{i\in I}  J_i = \{1,\dots,k\} } } (-1)^{|I|}
\end{align*} 
(i) If $J^*$ is a proper subset of $\{1,\dots, k\}$ then there are no terms in the sum.

\noindent (ii)\ If  $\cup_{\ell\in I}  J_\ell = \{1,\dots,k\}$ then each $j_i\in \cup_{\ell \in I}  J_\ell$, so we conclude that $i\in I$. Therefore $I=\{0,\dots,d\}$ and the result follows.
\end{proof}

 \subsection{Explicitly enumerating the coefficients of  $P_g(T)$}
It turns out that the quantities $\Lambda_j$ and $u_{j,i}$ also feature in the Khovanskii polynomial itself. Indeed, expanding the polynomial $P_g(T)$ we find that the leading two terms of $P_g(T)$ are
 \begin{align*}
\frac{1}{d!}   \sum_{\substack{ J\subset \{ 1,\dots,k\}\\ |J|\geqslant 1}} (-1)^{|J|-1}& (T^{d}-d(N_J-\tfrac {d+1}2) T^{d-1})
\\ & =
\frac{T^{d}}{d!} +\frac 12 \frac{(d+1)T^{d-1}}{(d-1)!} - \frac{T^{d-1}}{(d-1)!}  \bigg(\min_{1\leqslant j\leqslant k} \Delta_j +\sum_{i=1}^d
\min_{1\leqslant j\leqslant k} u_{j,i} \bigg)
 \end{align*}
 since $N_J$ is a sum of various maximums and we have the identity
\begin{equation}
\label{eq: supercool identity}
 \sum_{\substack{ J\subset \{ 1,\dots,k\}\\ |J|\geqslant 1}} (-1)^{|J|}  \max_{j\in J} a_j=-\min_{1\leqslant j\leqslant k} a_j
\end{equation}
 for any sequence $\{ a_j\}$. The proof of \eqref{eq: supercool identity} is an exercise in inclusion-exclusion. \\


\section{Delicate linear algebra and an effective Khovanskii's theorem}\label{6}

The proof of Theorem \ref{thm:effectiveKhovanksii} rests on various principles of quantitative linear algebra. The first is an application of the pigeon-hole principle.

\begin{Lemma}\label{SL} Let $M$ be a non-zero $m$-by-$n$ matrix with integer coefficients and $n > m$. Let $K$ be the maximum of the absolute values of the entries of $M$. Then there is a solution to $MX =0$ with $X \in \mathbb{Z}^n \setminus \{0\}$ and $$||X||_{\infty} \leqslant (Kn)^m.$$ 
\end{Lemma}
\noindent To prove Corollary~\ref{Corollary: cor to Siegel} in the next section, we will need the more sophisticated   Siegel's lemma  due to Bombieri--Vaaler \cite{BV83}, which gives a basis for $\ker M$ rather than just a single vector $X$; for the results in this section, the elementary result in Lemma \ref{SL} suffices.
\begin{proof}
Suppose first that $Kn$ is odd. If there were two distinct vectors $X_1,X_2 \in \mathbb{Z}^n$ for which $MX_1 = MX_2$ and $\Vert X_i\Vert_\infty \leqslant \frac{1}{2}(Kn)^m$, then by choosing $X = X_1 - X_2$ we would be done. Now, the number of vectors $X \in \mathbb{Z}^n$ for which $\Vert X\Vert_{\infty} \leqslant \frac{1}{2} (Kn)^m$ is equal to $(2(\frac{1}{2}((Kn)^m - 1)) + 1)^n$, which is $(Kn)^{mn}$. For all such $X$ we have $\Vert MX\Vert_{\infty} \leqslant \frac{1}{2} (Kn)^{m+1}$ and $MX \in \mathbb{Z}^m$. We may further assume that $MX \neq 0$, since otherwise we would be immediately done. There are exactly $(2(\frac{1}{2}((Kn)^{m+1} - 1)) + 1)^m$ vectors $Y \in \mathbb{Z}^m \setminus \{0\}$ with $\Vert Y\Vert_{\infty} \leqslant \frac{1}{2} (Kn)^{m+1}$, i.e exactly $(Kn)^{m(m+1)} - 1$ such vectors. Since $n\geqslant m+1$, by the pigeonhole principle we may find distinct $X_1,X_2$ with $MX_1 = MX_2$ as required. 

If $Kn$ is even, then the number of vectors $X \in \mathbb{Z}^n$ for which $\Vert X\Vert_{\infty} \leqslant \frac{1}{2} (Kn)^m$ is exactly $((Kn)^{m} + 1)^n$, and there are at most $((Kn)^{m+1} + 1)^m - 1$ vectors $Y \in \mathbb{Z}^m \setminus \{0\}$ with $\Vert Y\Vert_{\infty} \leqslant \frac{1}{2} (Kn)^{m+1}$. Since \[ ((Kn)^m + 1)^n > ((Kn)^{m+1} + 1)^m - 1,\] we can conclude using the pigeonhole principle as before.
\end{proof}

Next, we will consider solutions to the equation $My = b$ in which all the coordinates of $y$ are positive integers. 

\begin{Lemma}
\label{lemma more specific estimate}
Let $M = (\mu_{ij})_{i \leqslant m, \, j \leqslant n}$ be an $m$-by-$n$ matrix with integer coefficients, with $m \leqslant n$ and $\rank M = m$, and let $b \in \mathbb{Z}^m$. Suppose that $\max_{i,j} \vert \mu_{ij}\vert \leqslant K_1$ and $\Vert b\Vert_\infty \leqslant K_2$ (where we choose $K_1,K_2 \geqslant 1$), and suppose that there is some $x \in \mathbb{Z}_{> 0}^n$ for which $Mx = b$. Then we may find $y \in \mathbb{Z}_{> 0}^n$ for which $M y = b$ and \[\Vert y\Vert_\infty \leqslant (n-m) (n^{m}m^m K_1^{2m} + m^m K_1^m) +  m^m K_1^{m-1}K_2.\] 
\end{Lemma}

\begin{proof}
We prove this by induction on $n$. The base case is $n=m$. In this case we observe that $M$ is invertible, and so $x = y = M^{-1} b$. Using the formula $M^{-1} = \det(M)^{-1} \operatorname{adj}(M)$, and since $(\det M)^{-1} \leqslant 1$ as $M$ has integer entries, we conclude that $\Vert y\Vert_\infty \leqslant m!K_1^{m-1} K_2 \leqslant m^m K_1^{m-1} K_2$. This gives the base case. 

We proceed to the induction step, assuming that $n \geqslant m+1$. By Lemma~\ref{SL}, there is a vector $X \in \mathbb{Z}^n \setminus \{0\}$ such that $MX = 0$ and \[\Vert X\Vert_{\infty}\leqslant n^{m} K_1^m.\] Replacing $X$ by $-X$ if necessary, we may assume that $X$ has at least one positive coordinate with respect to the standard basis. Let $S \subset \{1,\dots,n\}$ be the set of indices where the coordinate of $X$ is positive.

Take $x$ from the hypotheses of the lemma, and write $x = (x_1,\dots,x_n)^T \in \mathbb{Z}_{> 0}^n$. By replacing $x$ with $x - \lambda X$ for some $\lambda \in \mathbb{Z}_{> 0}$ as appropriate, we may assume that there is some $i \in S$ for which $1 \leqslant x_i \leqslant \Vert X\Vert_\infty + 1 \leqslant n^m K_1^m + 1$. Fix such an $i$ and $x_i$, and now consider the $m$-by-$(n-1)$ matrix $M^{\{i\}}$ which is $M$ with the $i^{th}$ column removed. Similarly define $x^{\{i\}} \in \mathbb{Z}_{> 0}^{n-1}$ to be the vector $x$ with the $i^{th}$ coordinate removed. Then \[ M^{\{i\}}x^{\{i\}} = b - M(x_i e_i),\] where $e_i$ is the $i^{th}$ standard basis vector in $\mathbb{R}^n$. 

Observe that $b - M(x_i e_i) \in \mathbb{Z}^m$ with 
\[\Vert b - M(x_i e_i)\Vert_\infty\leqslant K_2 + K_1 x_i \leqslant K_2 + K_1(1 + n^mK_1^m) \leqslant n^m K_1^{m+1} + K_1 + K_2 .\] Now $\rank M^{\{i\}}$ is either $m$ or $m-1$. If $\rank M^{\{i\}} = m$ then, by the induction hypothesis (with $x$ replaced by $x^{\{i\}}$), there is some $y^{\{i\}} \in \mathbb{Z}_{> 0}^{n-1}$ for which $M^{\{i\}}y^{\{i\}} = b - M(x_i e_i)$ and 
\begin{align*}
\Vert y ^{\{i\}} \Vert_\infty &\leqslant (n-m-1)((n-1)^m m^m K_1^{2m} + m^m K_1^m) + m^m K_1^{m-1} (n^m K_1^{m+1} + K_1 + K_2) \\
& \leqslant (n-m-1)(n^m m^m K_1^{2m} + m^m K_1^m) +  m^m K_1^{m-1} (n^m K_1^{m+1} + K_1 + K_2) \\
& = (n-m)(n^m m^m K_1^{2m} + m^m K_1^m) + m^m K_1^{m-1} K_2.
\end{align*} Let $y: = y^{\{i\}} + x_i e_i$, where we have abused notation by treating $y^{\{i\}}$ also as an element of $\mathbb{Z}_{\geqslant 0}^n$ by extending by $0$ in the $i^{th}$ coordinate. Then we have $y \in \mathbb{Z}_{>0}^n$, $My = b$, and  \[ \Vert y\Vert_\infty \leqslant \max(\Vert y ^{\{i\}} \Vert_\infty, n^m K_1^m + 1) \leqslant (n-m)(n^m m^m K_1^{2m} + m^m K_1^m) + m^m K_1^{m-1} K_2\] since $n \geqslant m+1$. Thus we have completed the induction in this case.  

If $\rank M^{\{i\}} = m-1$ then there are some further cases. If $m=1$ and $M^{\{i\}}$ is the zero matrix, then we can choose any vector $y^{\{i\}} \in \mathbb{Z}^{n-1}_{>0}$. Otherwise, we may replace $M^{\{i\}}$ with $m-1$ of its rows. Call this new $(m-1)$-by-$(n-1)$ matrix $M_{\operatorname{res}}^{\{i\}}$, and further we may assume that $\rank M_{\res}^{\{i\}} = m-1$. Denote the analogous restriction of the vector $b - M(x_i e_i)$ as $b_{\operatorname{res}} - M(x_i e_i)_{\res}$. Then by the induction hypothesis as applied to $M_{\res}^{\{i\}}$, there is some $y^{\{i\}} \in \mathbb{Z}_{>0}^{n-1}$ for which $M^{\{i\}}_{\res} y^{\{i\}} = b_{\operatorname{res}} - M(x_i e_i)_{\res}$ and 
\begin{align*}
\Vert y^{\{i\}}\Vert_{\infty} &\leqslant (n-m)(n^{m-1} m^{m-1} K_1^{2m-2} + m^{m-1} K_1^{m-1}) + m^{m-1} K_1^{m-2}(n^{m-1} K_1^m + K_1 + K_2) \\
&\leqslant (n-m)(n^m m^m K_1^{2m} + m^m K_1^m) + m^m K_1^{m-1} K_2
\end{align*}
\noindent since $m \geqslant 2$, thus completing the induction as above. 
\end{proof}

\begin{Corollary}\label{6.1}
Let $M = (\mu_{ij})_{i \leqslant m, \, j \leqslant n}$ be an $m$-by-$n$ matrix with integer coefficients, and let $b \in \mathbb{Z}^m$. Suppose that $\max_{i,j} \vert \mu_{ij}\vert \leqslant K_1$ and $\Vert b\Vert_\infty \leqslant K_2$ (where we choose $K_1,K_2 \geqslant 1$), and suppose that there is some $x \in \mathbb{Z}_{> 0}^n$ for which $Mx = b$. Then we may find $y \in \mathbb{Z}_{> 0}^n$ for which $M y = b$ and \[\Vert y\Vert_\infty \leqslant 2n^{m+1} m^m K_1^{2m} + m^m K_1^{m-1} K_2.\] 
\end{Corollary}

\begin{proof}
We restrict $M$ to a maximal linearly independent subset of its rows and so obtain an $m'$-by-$n$ matrix $M'$ with   $\rank M' = m'\leqslant n$.
The result   follows by applying  Lemma \ref{lemma more specific estimate} to $M'$. 
\end{proof}

We introduce a partial ordering $<_{\unif}$ on $\mathbb{Z}^d$ by saying that $x \leq_{\unif} y$ if $x_i \leqslant y_i$ for all $i \leqslant d$ (that is, $y-x\in \mathbb{Z}_{\geqslant 0}^d$ as in the Mann-Dickson lemma). The next lemma controls the set of minimal solutions (with respect to the partial ordering $<_{\unif}$) to a certain kind of linear equation.

\begin{Lemma}\label{6.2}
Let $n = n_1 + n_2 \geqslant 2$  with $n_1,n_2 \in \mathbb{Z}_{> 0}$. Let $M = (\mu_{ij})_{i \leqslant m, \, j \leqslant n}$ be an $m$-by-$n$ matrix with integer coefficients, and $b \in \mathbb{Z}^m$. Suppose that $\max_{i,j} \vert \mu_{ij}\vert \leqslant K_1$ and $\Vert b\Vert_\infty \leqslant K_2$ (where we choose $K_1,K_2 \geqslant 1$). Let 
\[
S = S(M,b,n_1,n_2) = \bigg\{ \begin{pmatrix} x\\ y\end{pmatrix} \in \mathbb{Z}_{> 0}^{n_1} \times \mathbb{Z}_{> 0}^{n_2}: M\begin{pmatrix} x\\ y\end{pmatrix} = b\bigg\},
\] 
and let $S_{\min} = S_{\min}(M,b,n_1,n_2)$ be defined as
\[
S_{\min} = \bigg\{x \in \mathbb{Z}_{> 0}^{n_1}: \exists y \in \mathbb{Z}_{> 0}^{n_2} \text{ with }  \begin{pmatrix} x\\ y\end{pmatrix} \in S \text{ and } \not \exists \,  \begin{pmatrix} x_*\\ y_*\end{pmatrix}   \in S \text{ with } x_* <_{\unif} x\bigg\}.
\] 
If $x_{\min} \in S_{\min}$ then \[ \Vert x_{\min}\Vert_\infty \leqslant 2^{2n} m^{mn} K_1^{m(n+3)} n^{m+1} + 2^n m^{mn} K^{mn}_1 K_2.\] 
\end{Lemma}

\begin{proof}
We use induction on $n_1$.  If $S$ is empty then Lemma~\ref{6.2} is vacuously true. Otherwise   $S$ is non-empty and so is $S_{\min}$.

If $n_1 = 1$ then $S \subset \mathbb{Z}_{>0}$ so $\vert S_{\min}\vert = 1$ by the well-ordering principle. Writing $S_{\min}=\{x_{\min}\}$ we
note that there exists  $ (\begin{smallmatrix} x\\ y\end{smallmatrix}) \in S$ by Corollary \ref{6.1} with $x\leqslant 2n^{m+1} m^m K_1^{2m} + m^m K_1^{m-1} K_2$, and so $ x_{\min}  \leqslant x\leqslant 2n^{m+1} m^m K_1^{2m} + m^m K_1^{m-1} K_2$.  

If $n_1 \geqslant 2$ let $x_{\min} \in S_{\min}$ and choose  $y \in \mathbb{Z}^{n_2}_{>0}$ with $ (\begin{smallmatrix} x_{\min}\\ y\end{smallmatrix})  \in S$. By Corollary~\ref{6.1} we may choose $(\begin{smallmatrix} x_*\\ y_*\end{smallmatrix})  \in S$ with $\Vert x_* \Vert_{\infty} \leqslant  2n^{m+1} m^m K_1^{2m} + m^m K_1^{m-1} K_2$. Thus there is some $i \leqslant n_1$ for which \[\vert x_{\min,i}\vert \leqslant  2n^{m+1} m^m K_1^{2m} + m^m K_1^{m-1} K_2,\] as otherwise $x_* <_{\unif} x_{\min}$, in contradiction to the fact that $x_{\min} \in S_{\min}$. Fixing such a coordinate $i$, as in the proof of Lemma~\ref{lemma more specific estimate} we let $x_{\min}^{\{i\}}$ denote the vector $x_{\min}$ with the $i^{th}$ coordinate removed, and let $M^{\{i\}}$ be the matrix $M$ but with the $i^{th}$ column removed (from the initial set of $n_1$ columns). Then \[ M^{\{i\}}\begin{pmatrix}x_{\min}^{\{i\}} \\ y \end{pmatrix} = b - M \begin{pmatrix} x_{\min,i} e_i \\ 0\end{pmatrix},\]  where $e_i$ is the $i^{th}$ basis vector in $\mathbb{R}^{n_1}$. We have 
\begin{align*}
\Vert b - M((x_{\min,i} e_i, 0)^T)\Vert_\infty &\leqslant K_2 + K_1\vert x_{\min,i}\vert  \\
 &\leqslant K_2 + K_1( 2n^{m+1} m^m K_1^{2m} + m^m K_1^{m-1} K_2) \\
&= 2n^{m+1} m^m K_1^{2m+ 1} + (m^m K_1^{m} + 1) K_2\\
& \leqslant 2n^{m+1} m^m K_1^{2m+ 1} + 2m^m K_1^{m} K_2.
\end{align*} 

The vector $x^{\{i\}}_{\min} \in \mathbb{Z}_{> 0}^{n_1 - 1}$ is in $S_{\min}(M^{\{i\}}, b - M(\begin{smallmatrix} x_{\min,i} e_i \\ 0 \end{smallmatrix}),n_1-1,n_2)$. Indeed, were there another vector $(w,z)^T \in \mathbb{Z}_{> 0}^{n_1 - 1} \times \mathbb{Z}_{> 0}^{n_2}$ with $(w,z)^T \in S(M^{\{i\}}, b - M(\begin{smallmatrix} x_{\min,i} e_i \\ 0 \end{smallmatrix}),n_1-1,n_2)$ and $w <_{\unif} x_{\min}^{\{i\}}$, then $w + x_ie_i <_{\unif} x_{\min}$ and $(\begin{smallmatrix} w + x_{\min,i}e_i \\ z \end{smallmatrix})\in S(M,b,n_1,n_2)$, contradicting the minimality of $x_{\min}$. (We have abused notation here by treating $w$ as both an element of $\mathbb{Z}_{> 0}^{n-1}$ and, by extending by $0$, an element of $\mathbb{Z}_{\geqslant 0}^{n}$.) So by the induction hypothesis we have 
\begin{align*}
&\Vert x_{\min}^{\{i\}}\Vert_\infty \\
&\leqslant 2^{2(n-1)} m^{m(n-1)} K_1^{m(n+2)} n^{m+1} + 2^{n-1} m^{m(n-1)} K_1^{m(n-1)}(2n^{m+1} m^m K_1^{2m+ 1} + 2m^m K_1^{m} K_2)\\
&\leqslant 2^{2n} m^{mn} K_1^{m(n+3)} n^{m+1} + 2^n m^{mn} K^{mn}_1 K_2.
\end{align*}

So \[ \Vert x_{\min}\Vert_\infty = \max( \Vert x_{\min}^{\{i\}}\Vert_{\infty}, \vert x_{\min,i}\vert) \leqslant2^{2n} m^{mn} K_1^{m(n+3)} n^{m+1} + 2^n m^{mn} K^{mn}_1 K_2\] too, and the induction is completed. \end{proof}

We are now ready to prove an effective version of Khovanskii's theorem. Our method is a quantitative adaptation of Nathanson--Ruzsa's argument from \cite{NR02}. 
\begin{proof}[Proof of Theorem \ref{thm:effectiveKhovanksii}]
Without loss of generality, we may first translate $A$ (which preserves the width $w(A)$) so that $0 \in A$. Therefore we can assume that $\max_{a \in A} \Vert a\Vert_{\infty} \leqslant w(A)$. We can also assume that $A-A$ contains $d$ linearly independent vectors: If not we can project the question down to a smaller dimension (by removing some co-ordinate but keeping all the linear dependencies) and the result follows by induction on $d$. So $|A|=:\ell\geqslant d+1$.

Let us now recall the lexicographic ordering on $\mathbb{Z}^d$. If $x = (x_1,\dots,x_d)^T \in \mathbb{Z}^d$ and $y = (y_1,\dots,y_d)^T \in \mathbb{Z}^d$ we say that $x<_{\lex} y$ if there exists some $i\leqslant d$ for which $x_i < y_i$ and $x_j = y_j$ for all $j <i$. This is a total ordering on $\mathbb{Z}^d$. 

Following Nathanson--Ruzsa, we say that an element $x \in \mathbb{Z}_{ \geqslant 0}^\ell$ is \emph{useless} if there exists $y \in \mathbb{Z}_{ \geqslant 0}^\ell$ with $y <_{\lex} x$, $\Vert y\Vert_1 = \Vert x\Vert_1$ and $\sum_{i \leqslant \ell} x_j a_j = \sum_{j \leqslant \ell} y_j a_j$. We say that element $x \in \mathbb{Z}_{ \geqslant 0}^\ell$ is \emph{minimally useless} if there does not exist a useless $x^\prime \in \mathbb{Z}_{\geqslant 0}^\ell$ for which $x^\prime <_{\unif} x$. Let $U$ denote the set of useless elements and $U_{\min}$ be the set of minimally useless elements. By definition see that 
\[ 
U = \bigcup_{u \in U_{\min}} \{ x \in \mathbb{Z}_{\geqslant 0}^\ell: x \geqslant_{\unif} u\}.
\] 

For $x \in U_{\min}$,  let $I_1 = \{i \leqslant \ell: x_i \geqslant 1\}$ and $I_2 = \{j \leqslant \ell: y_j \geqslant 1\}$ (with $y$ as above).
Now $I_1 \cap I_2 = \emptyset$ else if $i \in I_1 \cap I_2$ then $x - e_i$ is also useless (via $y - e_i$) contradicting minimality.
We may assume that both $I_1$ and $I_2$ are non-empty, since otherwise we would have $x = y = 0$.
Evidently $\min I_1 < \min I_2$ as  $y < _{\lex} x$.

By the Mann-Dickson lemma we know that $U_{\min}$ is finite, but now we will be able to get an explicit bound on $\max(\Vert u\Vert_\infty: u \in U_{\min})$: 

Fix a pair of disjoint non-empty subsets $I_1\cup I_2 \subset \{1,\dots, \ell\}$ with   $\min I_1 < \min I_2$,  and let $n_1 = \vert I_1\vert$, $n_2 = \vert I_2 \vert$, with $n =n_1 + n_2 \leqslant \ell$.  We define a $(d+1)$-by-$n$ matrix $M$ where the columns are indexed by the elements of $I_1\cup I_2$, and the row numbers run from $0$ to $d$.
If $j\in I_1$ then $M_{0,j}=1$ and $M_{i,j}=(a_j)_i$ for $1\leqslant i\leqslant d$; if $j\in I_2$ then $M_{0,j}=-1$ and $M_{i,j}=-(a_j)_i$ for $1\leqslant i\leqslant d$.
Then the top row of the equation $M (\begin{smallmatrix} x\\ y\end{smallmatrix})=0$ with $x \in \mathbb{Z}_{> 0}^{n_1}$ and $y \in \mathbb{Z}_{> 0}^{n_2}$
gives that $\Vert y\Vert_1 = \Vert x\Vert_1$ and the 
$i^{th}$ row yields that $\sum_{j \leqslant \ell} x_j (a_j)_i= \sum_{j \leqslant \ell} y_j (a_j)_i$ for $1 \leqslant i \leqslant d$, so together they yield that 
$\sum_{j \leqslant \ell} x_j a_j = \sum_{j \leqslant \ell} y_j a_j$. 
 
 By the minimality of $x$ there cannot exist $(x_*, y_*) \in \mathbb{Z}_{> 0}^{n_1} \times \mathbb{Z}_{> 0}^{n_2}$ such that $M (\begin{smallmatrix} x_*\\ y_*\end{smallmatrix}) =  0 $ and $x_* <_{\unif} x$. Indeed, by construction of $I_1$ and $I_2$ we would have (after extending by zeros) that $y_* <_{\lex} x_*$, thus implying that $x_*$ is useless -- contradicting the fact that $x$ is minimally useless. 

Using Lemma \ref{6.2}, as applied to the matrix $M$ with $K_1 = \max_{a \in A} \Vert a\Vert_{\infty}:= K$ and $K_2 = 1$, we conclude that 
\begin{align*}
 \Vert x\Vert_\infty &\leqslant 2^{2\ell}(d+1)^{\ell(d+1)} \ell^{d+2}K^{(d+1) (\ell+3)} + 2^\ell (d+1)^{\ell(d+1)} K^{\ell(d+1)}\\
 &\leqslant 2^{2\ell+1}(d+1)^{\ell(d+1)} \ell^{d+2}K^{(d+1) (\ell+3)}  
 \end{align*}

In \cite[Lemma 1]{NR02}, Nathanson and Ruzsa proved that for all $U^\prime \subset U_{\min}$ \[ B(N,U^{\prime}): = \vert \{x \in \mathbb{Z}_{ \geqslant 0}^s: \Vert x\Vert_1 = N, \, x \geqslant_{\unif} u \text{ for all } u \in U^\prime\}\vert\] is equal to a fixed polynomial in $N$, once $N \geqslant \ell \max_{u \in U^\prime} \Vert u \Vert_\infty$. Indeed, let $U^\prime = \{u_1,\dots,u_m\}$, where each $u_j = (u_{1,j}, u_{2,j}, \dots, u_{s,j}) \in \mathbb{Z}_{\geqslant 0}^\ell$. Letting $u_i^* = \max_{j \leqslant m} u_{i,j}$, and \\$u^* = (u_1^*, u_2^*, \dots, u_\ell^*)$, we have that 
\begin{align*}
B(N,U^{\prime}) &= \vert \{ x \in \mathbb{Z}_{\geqslant 0}^\ell: \Vert x\Vert_1 = N, \, x \geqslant_{\unif} u^*\}\vert \\
& = \vert \{ x \in \mathbb{Z}_{\geqslant 0}^\ell: \Vert x\Vert_1 = N - \Vert u^*\Vert_1\}\vert \\
& = \left( \begin{matrix} 
N - \Vert u^*\Vert_1 + \ell - 1 \\
\ell-1 \end{matrix} \right)
\end{align*}
provided $N \geqslant \Vert u^*\Vert_1$, which is a polynomial in $N$. Since $\Vert u^* \Vert_1 \leqslant \ell \max_{u \in U^\prime} \Vert u\Vert_\infty$, our claim follows. 

Then by inclusion-exclusion we have
\begin{align*}
\vert  NA\vert &= \vert \{ x \in \mathbb{Z}_{ \geqslant 0}^\ell: \Vert x\Vert_1 = N, \, x \text{ is not useless}\}\vert  \\
& =  \sum\limits_{U^\prime \subset U_{\min}}(-1)^{\vert U^\prime\vert} B(N,U^\prime)
\end{align*}
which is a polynomial in $N$ once $N \geqslant N_{\text{Kh}}(A)$ where 
\[
N_{\text{Kh}}(A)\leqslant 2^{2\ell+1}(d+1)^{\ell(d+1)} \ell^{d+3}K^{(d+1) (\ell+3)} \leqslant (2\ell w(A))^{(d+4)\ell},
\]
as  $K: =\max_{a \in A} \Vert a\Vert_{\infty} \leqslant w(A)$. To obtain the last displayed inequality we assumed that $d\geqslant 2$ (as we use Lemma \ref{Lemma: 1-dim} for $d=1$ which gives $N_{\text{Kh}}(A)\leqslant w(A) -1$), and $\ell\geqslant d+1$. 
\end{proof}

\section{Structure bounds in the general case: proof of Theorem \ref{thm main theorem} }\label{5}
\label{section the non-simplex case}

We start by introducing the central structural result of this section. As a reminder, we say that $p \in \ex(H(A))$ if  there is a vector $v \in \spn(A-A) \setminus \{0\}$ and a constant $c$ such that $\langle v , p \rangle = c$ and $\langle v , x \rangle > c$ for all $x \in H(A) \setminus\{p\}$.

\begin{Lemma}[Decomposing $\mathcal{P}(A)$]
\label{Lemma :  minimal elements sumset endgame}
Let $B \cup \{0\} \subset A \subset \mathbb{Z}^d$ with $\vert A\vert = \ell$ and $B$ a non-empty linearly independent set. Suppose that $0 \in \ex(H(A))$. Let $A^+$ denote the set of $x \in \mathcal{P}(A) \cap C_B$ with the property that, for all $b \in B$, $x -b \notin \mathcal{P}(A) \cap C_B$. Then $A^+$ has the following two properties: 
\begin{equation}
\label{decomp}
C_{B} \cap \mathcal{P}(A) = A^+ + \mathcal{P}(B \cup \{0\})
\end{equation} and \[ A^+ \subset NA\] for some $N \leqslant  N_0(A):=2^{11d^2} d^{12d^6} \ell^{3d^2} w(A)^{8 d^6}$.
\end{Lemma}
 
 Lemma \ref{Lemma :  minimal elements sumset endgame} is straightforward for large $N$ (\eqref{decomp} was already given in \cite[Proposition 4]{GS20}) but our focus is on getting an effective  bound on such $N$. 

The only other ingredient in the proof of Theorem \ref{thm main theorem} is the following classical lemma:
\begin{Lemma}[Carath\'{e}odory]
\label{lem: Caratheodory}
Let $A \subset \mathbb{R}^d$ be a finite set, and let $V: = \spn(A-A)$. If $\dim V = r$, then \[ H(A) = \bigcup \limits_{\substack{B\subset \ex(H(A)) \\ \vert B\vert = r+1 \\ \spn(B-B) = V}} H(B).\]
\end{Lemma}
\begin{proof}
After an affine transformation one may assume that $V = \mathbb{R}^d$. Then see \cite[Lemma 4]{GS20} for the proof, in which the union is taken over all $B \subset A$ with $\vert B\vert = d+1$ and $\spn(B-B) = \mathbb{R}^d$. The equality as claimed, where $B \subset \ex(H(A))$, then follows from general fact that $H(A) = H(\ex(H(A))$ (see Lemma \ref{Lemma Extremal points}).
\end{proof}

\begin{proof}[Proof of Theorem \ref{thm main theorem}]  Let 
\begin{equation}
\label{eq:inclusionhighdim}
v \in (NH(A) \cap (a_0N + \Lambda_{A-A})) \setminus (\bigcup\limits_{ b \in \ex(H(A))} (bN - \mathcal{E}(b - A))).
\end{equation}
We will show that $v \in NA$ for all $N\geqslant (d+1)N_0(A)$.

Let $V = \spn(A-A)$ and $r = \dim V$ as above. 
By Lemma \ref{lem: Caratheodory} there exists a set $B \subset \ex(H(A))$ with $\vert B\vert = r+1$ such that $v \in NH(B^*)$ and $\spn(B-B) = V$. Write $B = \{b_0,b_1,\dots,b_{r}\}$. Since $v \in NH(B)$ we can write $v = \sum_{i=0}^r c_i b_i$ for some real $c_i \geqslant 0$ such that $\sum_{i=0}^r c_i = N$. Since $N \geqslant (d+1)N_0(A)$ there must be some $c_{i} \geqslant N_0(A)$. After permuting coordinates, we will assume that $c_{r} \geqslant N_0(A)$. Thus 
\[
b_{r}N - v =   \sum\limits_{i=0}^{r-1} c_i(b_{r} - b_i) \in (N-N_0(A))\cdot H(b_{r} - B),
\] 
so that $b_{r}N - v \in C_{b_{r} - B} \subset C_{b_r - A}$. By the assumption \eqref{eq:inclusionhighdim} we also have $b_{r}N - v \notin \mathcal{E}(b_{r}-A)$ and $b_{r}N - v \in \Lambda_{b_r -A}$. Hence $b_{r} N - v \in \mathcal{P}(b_{r} - A)$. We may now apply Lemma \ref{Lemma : minimal elements sumset endgame} to the sets $b_{r} - A$ and $(b_{r} - B) \setminus \{0\}$; the hypotheses are satisfied since $b_{r} \in \ex(H(A))$ implies $0 \in \ex(H(b_{r} - A))$. Furthermore, $w(b_r - A) = w(A)$. We thus obtain \[b_{r}N - v \in C_{b_{r} - B} \cap \mathcal{P}(b_{r}-A) = A^+ + \mathcal{P}(b_{r}- B^*)\] for some set $A^+ \subset C_{b_{r} - B} \cap \mathcal{P}(b_{r} - A)$ with $A^+ \subset N_0(A)(b_{r} - A)$.

Now let us write \[b_{r}N - v = u + w,\] with $u \in A^+$ and $w \in \mathcal{P}(b_{r} - B)$. Thus $u+w = \sum_{i=0}^{r-1} c_i(b_r - b_i)$, with $c_i \in \mathbb{R}_{ \geqslant 0}$ for all $i$ and $\sum_{i=0}^{r-1}c_i \leqslant N-N_0(A)$. Expressing $u$ and $w$ with respect to the basis $(b_r - B) \setminus \{0\}$, and noting that $u \in A^+ \subset C_{b_{r} - B}\cap N_0(A)(b_{r} - A)$, we infer that $w = \sum_{i=0}^{r-1}\gamma_i(b_{r} - b_i)$ with $\gamma_i \leqslant c_i$ and $\gamma_i \in \mathbb{Z}_{\geqslant 0}$ for all $i$. Hence $w \in (N-N_0(A))(b_r - B)$. 

Putting everything together we have 
\begin{align*}
 b_rN - v& =u+w \in N_0(A)(b_r - A) + (N-N_0(A))(b_r - B) \\ &\subset N_0(A)(b_{r} - A) + (N-N_0(A))(b_{r} - A)  = N(b_{r} - A).
 \end{align*}
  Hence $v \in NA$ as required. 
  
  The proof shows that $N_{\text{Str}}(A)\leqslant (d+1)N_0(A)=(d+1)2^{11d^2} d^{12d^6} \ell^{3d^2} w(A)^{8 d^6}\leqslant (d\ell\, w(A))^{13 d^6}$
  as we may take $d\geqslant 2$ (after Lemma \ref{Lemma: 1-dim}) .
\end{proof}

It remains to prove Lemma \ref{Lemma :  minimal elements sumset endgame}. The condition $x \in \mathcal{P}(A) \cap C_B$  but $x - b \notin \mathcal{P}(A) \cap C_B$  in the definition of $A^+$ is a minimality-type condition on $x$.\footnote{In fact this intuition can be phrased precisely: viewing $\mathcal{P}(A) \cap C_B$ as a poset $P$, where $x \leqslant_P y$ if $y - x \in \mathcal{P}(B) \cup \{0\}$, the set $A^+$ is exactly the minimal elements of this poset.} As our argument for analysing the set $A^+$ will not stay within $C_B$, it turns out to be convenient to separate the $\mathcal{P}(A)$ part and the $C_B$ part of this condition; this motivates the following definition. 

\begin{Definition}[Absolutely $B$-minimal]
Let $B \cup \{0\} \subset A \subset \mathbb{Z}^d$, with $A$ finite. We say that $u \in \mathcal{P}(A)$ is \emph{absolutely $B$-minimal} with respect to $A$ if $u- b \notin \mathcal{P}(A)$ for all $b \in B$. Let $\mathcal{S}_{\abs}(A,B)$ denote the set of absolutely $B$-minimal elements. 
\end{Definition}
\noindent Let   $\mathcal{S}_{\abs}(A,\emptyset) = \mathcal{P}(A)$ and use the convention that $C_{\emptyset} = \{0\}$.  By this definition $\mathcal{S}_{\abs}(A,B) \subset \mathcal{S}(A,B)$, though these sets needn't be equal, so
 being a $B$-minimal element is a weaker condition than being an absolutely $B$-minimal element.

For a subset $U \subset \mathbb{R}^d$ and $x \in \mathbb{R}^d$, we define 
 \[ \dist(x,U) : = \inf \limits_{ u \in U} \Vert x - u\Vert_\infty.\]

\begin{Lemma}[Controlling the absolutely $B$-minimal elements]
\label{Lemma: controlling absolute B minimal}
Let $B\cup \{0\} \subset A \subset \mathbb{Z}^d$ with $\vert A\vert = \ell \geqslant 2$, and assume that $B$ is a (possibly empty) linearly independent set. Let $r: = \dim \spn(A)$ and suppose that $0 \in \ex(H(A))$. If $x \in \mathcal{S}_{\abs}(A,B)$ and $\dist(x, C_B) \leqslant X$ then $x \in NA$ for some $N \in \mathbb{Z}_{> 0}$ with \[N \leqslant (X+1) 2^{10dr} d^{11d^5 r} \ell^{3dr} w(A)^{7d^5 r}.\] 
\end{Lemma}
\noindent Lemma \ref{Lemma: controlling absolute B minimal} is the main technical result of this section. The hypotheses allow $r$ to be less than $d$, even though we will only apply the lemma when $r=d$, since our proof involves induction on   $r$. Similarly, we do not assume that $\Lambda_A = \mathbb{Z}^d \cap \spn(A)$, as this property would not necessarily be preserved by the induction step. 
Deducing Lemma \ref{Lemma :  minimal elements sumset endgame} is straightforward:

\begin{proof}[Proof of Lemma \ref{Lemma :  minimal elements sumset endgame}]
If $x \in A^+$ then we can partition $B = B^\prime \cup B^{\prime\prime}$ so that $b^\prime \in B^\prime$ implies $x-b^\prime \notin C_B$ and $b^{\prime\prime} \in B^{\prime\prime} $ implies $x - b^{\prime\prime} \in C_B \setminus \mathcal{P}(A)$.

Writing $x$ with respect to the basis $B$, we get 
\[
x = \ell + \sum_{b^{\prime\prime} \in B^{\prime\prime}} c_{b^{\prime\prime}} b^{\prime\prime} \text{ where } \ell = \sum_{b^\prime \in B^\prime} \ell_{b^\prime} b^\prime
\] 
with  $c_{b^{\prime\prime}} \geqslant 1$ for all $b^{\prime\prime} \in B^{\prime\prime}$ and  $\ell_{b^\prime} \in [0,1)$ for all $b^\prime \in B^{\prime}$. 

Since $\Vert \ell\Vert_\infty \leqslant d w(A)$, this implies that $\dist(x, C_{B^{\prime\prime}}) \leqslant d w(A)$. Furthermore, for all $b^{\prime\prime} \in B^{\prime\prime}$ we have $x - b^{\prime\prime}\notin \mathcal{P}(A)$. Hence $x \in \mathcal{S}_{\abs}(A,B^{\prime\prime})$. By Lemma \ref{Lemma: controlling absolute B minimal} as applied to $B^{\prime\prime}$ and $X = dw(A)$, we may conclude that $x \in NA$ for $N\geqslant N_0(A)$ as in Lemma \ref{Lemma :  minimal elements sumset endgame}.

To establish \eqref{decomp}, note that $A^+ + \mathcal{P}(B \cup \{0\}) \subset \mathcal{P}(A) \cap C_B$ by definition.
On the other hand  if $y \in \mathcal{P}(A) \cap C_B$ and  there exists some $b_1 \in B$ with $y - b_1 \in \mathcal{P}(A) \cap C_B$ then we replace $y$ by $y-b_1$.
We repeat this with $b_2,\dots$ until the process terminates, which it must do since the sum of the coefficients of $y$ with respect to the basis $B$ decreases by $1$ at each step. We are left with  $y-b_1-\dots-b_k \in A^+$ so that $y \in A^+ + \mathcal{P}(B \cup \{0\})$. 
\end{proof}

It remains   is to prove Lemma \ref{Lemma: controlling absolute B minimal}. 
Following the proofs in \cite{GW20, CG20}  we now show that in certain favourable circumstances, $\mathcal{S}_{\abs}(A,B)$ may be controlled in terms
of the Davenport constant of  $\mathbb{Z}^d / \Lambda_B$. However this is not used in our proof of Lemma \ref{Lemma: controlling absolute B minimal} (except when $d=1$) but, for reasons of motivation, it is helpful to  understand why this type of argument fails. 

\begin{Lemma}
\label{Lemma: non simplex Davenport constant lemma}
Let $B \cup \{0\} \subset A \subset \mathbb{Z}^d$, with $A$ finite and $B$ a basis of $\mathbb{R}^d$. Suppose that $C_A = C_B$. Let $\mathbb{Z}^d / \Lambda_B: = G$. Then $\mathcal{S}_{\abs}(A,B) \subset NA$, where $N = \max(1, D(G) - 1)$ and $D(G)$ is the Davenport constant of $G$. 
\end{Lemma}
\begin{proof}
Let $x \in \mathcal{S}_{\abs}(A,B)$, and assume that $x \neq 0$. Then write \[ x = a_1 + a_2 + \cdots + a_{N_A(x)}\] for some $a_i \in A$. If there were a subsum $\sum_{i \in I} a_i \equiv 0\text{ mod }\Lambda_B$, then since $C_A = C_B$ we would have  $\sum_{i \in I} a_i \in C_A \cap \Lambda_B \subset C_B \cap \Lambda_B$. But since $B$ is a basis of $\mathbb{R}^d$ we have $C_B \cap \Lambda_B =  \mathcal{P}(B) \cup \{0\}$, so $\sum_{i \in I} a_i \in \mathcal{P}(B) \cup \{0\}$. By minimality of $N_A(x)$ we also have $\sum_{i \in I} a_i \neq 0$. Therefore $x \in \mathcal{P}(A) + y$ for some non-zero $y\in \mathcal{P}(B)$, contrary to the assumption that $x \in \mathcal{S}_{\abs}(A,B)$. Hence $N_A(x) \leqslant \max(1, D(G) - 1)$, which also takes care of the $x=0$ case.  
\end{proof}
\noindent If $C_B$ is a strict subset of $C_A$ then the above argument doesn't necessarily work, as $\sum_{i \in I} a_i \equiv 0\text{ mod }\Lambda_B$ does not automatically imply that $ \sum_{i \in I}a_i \in \mathcal{P}(B) \cup \{0\}$: Indeed the key issue  is how an element $a_1 + \cdots + a_N = x \in \mathcal{P}(A) \cap C_B$ can have partial sums $\sum_{i \in I} a_i \notin C_B$.  \\

\subsection*{Sketch of our proof of Lemma \ref{Lemma: controlling absolute B minimal}}
The easy cases are $r=1$ (which follows from any of the existing literature \cite{Na72, WCC11, GS20, GW20}, or from Lemma \ref{Lemma: non simplex Davenport constant lemma}) and $B = \emptyset$ (which is dealt with in Lemma \ref{Lemma: controlling small elements} below). From these base cases, we will construct a proof by induction on $r$. We may assume, therefore, that $r \geqslant 2$ and $B$ is non-empty. For this sketch, we will also assume that $r=d$. 
There are three main phases to the induction step. 

$\bullet$\ We provide an extra restriction on the region of $\mathbb{R}^d$ where $\mathcal{S}_{\abs}(A,B)$ can lie, by showing that if $x \in \mathcal{S}_{\abs}(A,B)$ then $\dist(x, \partial(C_A)) \leqslant Y$, where $\partial(C_A)$ is the topological boundary of $C_A$ and $Y$ is some explicit bound.\footnote{If $r \leqslant d-1$ then one cannot use the topological boundary here, since $\partial(C_A) = C_A$ in this case, but this issue may be circumvented.} The bound $\dist(x, \partial(C_A)) \leqslant Y$ is a generalisation of a basic result from the one dimensional case -- the classical `Frobenius postage stamp' problem  -- in which the boundary of $C_A$ is just $\{0\}$ and one shows that the exceptional set $\mathcal{E}(A)$ is finite. Since $\partial(C_A)$ is a union of $d-1$ dimensional facets, there is some non-zero linear map $\alpha: \mathbb{R}^d \longrightarrow \mathbb{R}$ for which $\dist(x, \ker \alpha) \leqslant Y$. 

$\bullet$\ We combine  the distance condition from above with the hypotheses of Lemma \ref{Lemma: controlling absolute B minimal}, giving $\dist(x,C_B) \leqslant X$ and $\dist(x,\ker \alpha) \leqslant Y$.  In turn, we show that this implies $\dist(x, C_B \cap \ker \alpha) \leqslant f(X,Y,A)$ (for some explicit function $f(X,Y,A)$), by a quantitative linear algebra argument. For this part, one should have in mind the situation of two rays, both starting from the origin. If $x$ is in a neighbourhood of both rays separately, then $x$ will be in some neighbourhood of the origin. The size of this neighbourhood will be determined by the angle between the rays (the smaller the angle, the larger the neighbourhood). To study the general dimension version of this phenomenon we avoid talking explicitly about angles, relying instead on the existence of suitable bases of vectors with integer coordinates. 

Defining $B^\prime = B \cap \ker \alpha$ then $C_{B^\prime} = C_B \cap \ker \alpha$, and so we  establish that $\dist(x, C_{B^\prime}) \leqslant f(X,Y,A)$.

$\bullet$\  Let $A^\prime= A \cap \ker \alpha$. If $x$ is expressed as a sum $a_1+ \cdots + a_N$ with $a_i \in A$ for all $i$ then only finitely many of the $a_i$ are in $A \setminus A^\prime$. This is because $\alpha(x)$ is bounded, by the assumption $\dist(x, \ker \alpha) \leqslant Y$, and $\alpha(a)>0$ for all $a \in A \setminus A^\prime$, since $\ker \alpha$ is a separating hyperplane for $H(A)$. 

Now let $x^\prime$ be the subsum of $a_1 + \cdots + a_N$ coming just from those $a_i \in A^\prime$. One still has an upper bound on $\dist(x^\prime,C_{B^\prime})$, since $\Vert x-x^\prime\Vert_{\infty}$ is bounded. One may also show that $x^\prime \in \mathcal{S}_{\abs}(A^\prime, B^\prime)$. However, $\dim \spn (A^\prime) = \dim \ker \alpha = d-1 < r$, so by applying the induction hypothesis we conclude that $x^\prime \in N^\prime A^\prime$ for some explicit $N^\prime$. Adding on the elements of $A \setminus A^\prime$, of which there are boundedly many, we end up with $x \in NA$ for some other explicit $N$.\\

\subsection*{Phase 1: Quantitative details} We will prove the following. 
 \begin{Lemma}[Interior points are representable]
\label{Lemma:distancefromboundary}
Let $A \subset \mathbb{Z}^d$ be a finite set with $0 \in A$ and $\vert A\vert = \ell \geqslant 2$. There is a   constant $K_A$ such that if $x \in C_A \cap \Lambda_A$ and \[(x + [-K_A,K_A]^d) \cap \spn(A) \subset C_A,\] then $x \in \mathcal{P}(A)$. Moreover we may take
\[K_A =  4d^{d} \ell^{3d} w(A)^{3d}. \] 
\end{Lemma} 

The proof will be a quantitative adaptation of an argument of Khovanskii from his original paper (Proposition 1 of \cite{Kh92}, repeated as Lemma 1 of \cite{L17arxiv}).  

\begin{Lemma}[Quantitative representation of basis elements]
\label{Lemma:repbasis}
Let $A \subset \mathbb{Z}^d$ be a finite set with $\vert A\vert = \ell\geqslant 2$ and $0 \in A$. If $u \in \Lambda_A$ then there exists $(n_a(u))_{a \in A} \in \mathbb{Z}^A$ for which $u=\sum_{a \in A} n_a(u) a$ and \[\vert n_a(u)\vert \leqslant 2 d^{d} \ell^{d+1} w(A)^{2d} + d^d w(A)^{d-1}\Vert u\Vert_{\infty}\] for all $a \in A$. 
\end{Lemma}
\begin{proof}
We may assume that $u \neq 0$ else the result is trivial. Pick some $(x_a(u))_{a \in A} \in \mathbb{Z}^A$ for which $\sum_{a \in A} x_a(u) a = u$. Let $A^\prime:=\{ a\in A:\ x_a(u) \neq 0\}$ and   $\ell^\prime = \vert A^\prime\vert $. Let $M$ be the $d$-by-$\ell^\prime$ matrix $M$ whose columns are the vectors $\sign(x_a(u))a$ for $a \in A^\prime$. The   absolute values of the coefficients of $M$ are all $\leqslant \max_{a \in A} \Vert a\Vert_{\infty}\leqslant w(A)$. Since $x^\prime(u) := (\vert x_a(u)\vert)_{a \in A^\prime} \in \mathbb{Z}^{A^\prime}_{>0}$ satisfies $Mx^\prime(u) = u$, we may apply Corollary \ref{6.1} and conclude that there is some $y(u) \in \mathbb{Z}^{A^\prime}_{>0}$ for which $My(u) = u$ and $\Vert y\Vert_{\infty} \leqslant 2d^d (\ell^\prime)^{d+1} w(A)^{2d} + d^d w(A)^{d-1} \Vert u\Vert_{\infty}$. We have $u=\sum_{a \in A} n_a(u) a$ with
 $n_a(u) := \sign(x_a(u)) y_a(u)$ for  $a \in A^\prime$, and $n_a(u): = 0$ otherwise. 
\end{proof}

\begin{proof}[Proof of Lemma \ref{Lemma:distancefromboundary}]
Let \[U = \{u \in \Lambda_A: u = \sum_{a \in A} c_a a \text{ with } c_a \in [0,1) \text{ for all } a \in A\}.\] From Lemma \ref{Lemma:repbasis}, we may write $u = \sum_{a \in A} n_a(u) a$ for coefficients $n_a(u) \in \mathbb{Z}$ satisfying 
\begin{align*}
\vert n_a(u) \vert &\leqslant 2d^d \ell^{d+1} w(A)^{2d} + d^d w(A)^{d-1} \Vert u\Vert_{\infty} \\ 
&\leqslant  2d^d \ell^{d+1} w(A)^{2d} + d^d \ell w(A)^d\\
& \leqslant 3 d^d \ell^{d+1} w(A)^{2d}
\end{align*}  
since $\Vert u\Vert_{\infty}\leqslant \ell w(A)$.   We let \[ D = 1 + \max\limits_{\substack{u \in U \\ a \in A}} \vert n_a(u)\vert,\] and write $K_A := D\ell w(A)$. 

Suppose that $x \in \Lambda_A \cap C_A$ with $(x + [-K_A, K_A]^d) \cap \spn(A) \subset C_A$. By the construction of $K_A$, we have $x - D \sum_{a \in A} a \in C_A$. Therefore, we may write $x = \sum_{a \in A} \lambda_a a$ for some real coefficients $\lambda_a$ which satisfy $\lambda_a \geqslant D$ for all $a$. Then consider \[u: = x - \sum_{a \in A} \lfloor \lambda_a \rfloor a.\] We have $u \in U$, so writing $u = \sum_{a\in A} n_a(u) a$ we get $x = \sum_{a \in A}(\lfloor \lambda_a \rfloor + n_a(u)) a$. Since $\lfloor \lambda_a \rfloor + n_a(u) \in \mathbb{Z}_{ \geqslant 0}$ by the construction of $D$, this shows that $x \in \mathcal{P}(A)$, as required. 

The bound on $K_A$ follows from the bound $D \leqslant 4d^d \ell^{d+1} w(A)^{2d}$. 
\end{proof}

We use a classical result due to Bombieri--Vaaler for the more complicated pieces of quantitative linear algebra to come:

\begin{Lemma}[Siegel's lemma, Theorem 2 of \cite{BV83}]
\label{Lemma: Siegel's lemma}
With $n \geqslant m$ let $M$ be an $m$-by-$n$ matrix with integer entries. Then the equation $MX = 0$ has $n-m$ linearly independent integer solutions $X_j = (x_{j,1}, \cdots , x_{j,n}) \in \mathbb{Z}^n$ such that \[ \prod\limits_{j=1}^{n-m} \Vert X_j\Vert_\infty \leqslant D^{-1} \sqrt{\det (M M^T)},\] where $D$ is the greatest common divisor of the determinants of all the $m$-by-$m$ minors of $M$.
\end{Lemma}

\begin{Corollary}
\label{Corollary: cor to Siegel}
With $n \geqslant m$ let $M$ be an $m$-by-$n$ matrix with integer entries. Let $K$ be the maximum of the absolute values of the entries of $M$. Then the equation $MX = 0$ has $n-m$ linearly independent integer solutions $X_j = (x_{j,1}, \cdots , x_{j,n}) \in \mathbb{Z}^n$ such that  \[ \prod\limits_{j=1}^{n-m}\Vert X_j\Vert_\infty \leqslant (m!)^{1/2} n^{m/2} K^m. \]
\end{Corollary}
\begin{proof}
In Lemma \ref{Lemma: Siegel's lemma} we have $D \geqslant 1$ and, since the coefficients of $MM^T$ are at most $n K^2$ in absolute value, we have $\det(MM^T) \leqslant m!(nK^2)^m$. 
\end{proof}

In our application, Lemma \ref{Lemma:distancefromboundary} will be combined with the following result. This uses Siegel's lemma to construct normal vectors to separating hyperplanes of $C_A$. 

\begin{Lemma}[Finding a close point on the boundary]
\label{Lemma structure of boundary}
Let $A \subset \mathbb{Z}^d$ with $0 \in A$, $\vert A\vert = \ell\geqslant 2$ and $r = \dim \spn(A)$. Let $x \in C_A$, and suppose that there is some $y\in \spn(A) \setminus C_A$ for which $\Vert x-y\Vert_{\infty} \leqslant D$. Then there are $r-1$ linearly independent vectors $\{a_1,\dots,a_{r-1}\} \subset A$, a vector $z \in \spn(\{a_1,\dots,a_{r-1}\})$ for which $\Vert x - z\Vert_{\infty} \leqslant D$,
and a vector $v \in \mathbb{Z}^d \cap \spn(A) \cap \spn(\{a_1,\dots,a_{r-1}\})^{\perp}$ for which

\begin{enumerate} 
\item $\Vert v\Vert_{\infty} \leqslant d^{2d^2} w(A)^{d^2}$;
\item $\langle v,w\rangle \geqslant 0$ for all $w \in C_A$;
\item $\langle v,w\rangle >0$ for all $w \in C_A \setminus \spn(\{a_1,\dots,a_{r-1}\})$. 
\end{enumerate}
\end{Lemma}
\begin{proof}
Since $C_A$ is convex, we know there is some maximal $\rho \in (0,1)$ for which \[z: = x  + \rho(y-x) \in C_A.\] Certainly $\Vert x-z\Vert_{\infty} \leqslant D$. 

To prove the other properties, let $f: \mathbb{R}^d \longrightarrow \mathbb{R}^d$ be some linear isomorphism for which $f(\spn(A)) = \mathbb{R}^r \times \{0\}^{d-r}$. Letting $A^\prime = f(A)$ and $z^\prime = f(z)$, we also have $f(C_A) = C_{A^\prime}$. Abusing notation to neglect the final $d-r$ coordinates, we have $z^\prime \in \partial(C_{A^\prime})$ (since every neighbourhood of $z$ contains a point in $\spn(A) \setminus C_A$). The structure of $\partial(C_{A^\prime})$ is well-understood from the theory of convex polytopes, which we recall in Appendix \ref{Appendix convex sets} below. Indeed, by Lemma \ref{Lemma structure of convex hull} there is some non-zero linear map $\alpha: \mathbb{R}^r \longrightarrow \mathbb{R}$ for which $z^\prime \in \ker \alpha$ and $\alpha(a^\prime) \geqslant 0$ for all $a^\prime \in A^\prime$. Furthermore, $\ker \alpha$ is spanned by some linearly independent set $A^{\prime\prime} \subset A^\prime$ with $\vert A^{\prime\prime}\vert = r-1$. Letting $\{a_1,\dots,a_{r-1}\} = f^{-1}(A^{\prime\prime})$, we have $z \in \spn(\{a_1,\dots,a_{r-1}\})$. 

We finish by constructing $v$. By applying Corollary \ref{Corollary: cor to Siegel} to an $r$-by-$d$ matrix whose rows are element of $A$ that are a basis for $\spn(A)$, we can construct a basis $X_1,\dots,X_{d-r} \in \mathbb{Z}^d$ for $\spn(A)^{\perp}$ with $\Vert X_i\Vert_\infty \leqslant d^d w(A)^d$ for all $i$. Noting that $(\spn(A)^\perp)^\perp = \spn(A)$, we then apply Corollary \ref{Corollary: cor to Siegel} again to the $(d-1)$-by-$d$ matrix whose first $r-1$ rows consist of the vectors $a_1,\dots,a_{r-1}$ and whose final $d-r$ rows consist of the vectors $X_1,\dots,X_{d-r}$; this gives a non-zero vector $v \in \mathbb{Z}^d \cap \spn(A) \cap \spn(\{a_1,\dots,a_{r-1}\})^{\perp}$ with $\Vert v\Vert_{\infty} \leqslant d^d(d^d w(A)^d)^d \leqslant d^{2d^2} w(A)^{d^2}$.

Finally, let $\beta: \spn(A) \longrightarrow \mathbb{R}$ denote the linear map $w \mapsto \langle v, w\rangle$. The kernel of $\beta$ is exactly $\spn(\{a_1,\dots,a_{r-1}\})$ (since otherwise, writing $\mathbb{R}^d = \spn(A) \oplus \spn(A)^{\perp}$, we would get that all of $\mathbb{R}^d$ is orthogonal to $v$). Since the map $\beta_f: \mathbb{R}^{r} \longrightarrow \mathbb{R}$ given by $\beta_f(w^\prime) = \beta(f^{-1}(w^\prime))$ is a linear map with $\ker \beta_f = \ker \alpha$, we conclude that $\beta_f = \lambda \alpha$ for some non-zero $\lambda \in \mathbb{R}$. By replacing $v$ by $-v$ if necessary, we may assume that $\lambda >0$. Therefore $\beta_f(w^\prime) \geqslant 0$ for all $w^\prime \in C_{A^\prime}$, and hence $\beta(w) \geqslant 0$ for all $w \in C_A$, as desired.
\end{proof}

The next result deals with the $B = \emptyset$ case of Lemma \ref{Lemma: controlling absolute B minimal}. It is a generalisation to arbitrary dimension of a trivial observation from the one dimensional case, namely that if $A \subset \mathbb{Z}_{ \geqslant 0}$ with $\min A = 0$, and if $v \in \mathcal{P}(A)$, then $v \in NA$ for all $N \geqslant v$.

\begin{Lemma}[Controlling small elements]
\label{Lemma: controlling small elements}
Let $A \subset \mathbb{Z}^d$ with $\vert A\vert =\ell \geqslant 2$ and $0 \in \ex(H(A))$. If $v \in \mathcal{P}(A) \setminus \{0\}$ and $N \geqslant 2d^{11d^3} \ell^{d} w(A)^{5d^3}\Vert v\Vert_{\infty}$ then $v \in NA$. 
\end{Lemma}

\begin{proof}
Suppose that $\dim \spn(A) = r$. We start by constructing a linear isomorphism $f: \mathbb{R}^d \longrightarrow \mathbb{R}^d$ for which $f(A) \subset \mathbb{Z}^r \times \{0\}^{d-r}$. Indeed, if $r=d$ there is nothing to do. Otherwise, we take some elements $a_1,\dots, a_r \in A$ which form a basis of $\spn(A)$. Then, by applying Corollary \ref{Corollary: cor to Siegel} to the $r$-by-$d$ matrix whose rows are given by the vectors $a_i$, we have vectors $v_{r+1}, \dots, v_d \in \mathbb{Z}^d$ such that $ \mathcal{B}: = \{a_1,\dots,a_r,v_{r+1},\dots,v_d\}$ is a basis for $\mathbb{R}^d$ and $\Vert v_i\Vert_{\infty} \leqslant d^d w(A)^d$ for each $i$. 

Now let $M = (\mu_{i,j})_{i,j \leqslant d}$ denote the $d$-by-$d$ matrix whose inverse $M^{-1}$ has columns given by the vectors from $\mathcal{B}$. Thus $M$ is the change of basis matrix that maps elements of $\mathcal{B}$ to the standard basis vectors of $\mathbb{R}^d$. By Cramer's rule, we see that \[ \vert \mu_{i,j}\vert \leqslant d^d (d^d w(A)^d)^d \leqslant d^{2d^2} w(A)^{d^2}.\] Furthermore, $\mu_{i,j} \in \frac{1}{D} \mathbb{Z}$ where $D \in \mathbb{Z}$ with \[\vert D\vert = \det(M^{-1}) \leqslant d^d(d^d w(A)^d)^d \leqslant d^{2d^2} w(A)^{d^2}.\] Now let $f$ be the linear map given by matrix $DM$, and let $A^\prime = f(A)$. Then $0 \in \ex(H(A^\prime))$, $A^\prime \subset \mathbb{Z}^r \times \{0\}^{d-r}$, $\spn(A^\prime) = \mathbb{R}^r \times \{0\}^{d-r}$ and 
\begin{equation}
\label{eq: changing width}
w(A^\prime) \leqslant d^{4d^2} w(A)^{2d^2 +1} \leqslant d^{4d^2} w(A)^{3d^2}.
\end{equation}
\noindent Henceforth we will abuse notation and consider $A^\prime$ as a subset of $\mathbb{Z}^r$.

We now make an appeal to facts about $C_{A^\prime}$ and $\partial(C_{A^\prime})$ which are laid out in Lemma \ref{Lemma structure of convex hull} below. In particular, we see that there is a collection of non-zero linear maps $\alpha_1,\dots,\alpha_n: \mathbb{R}^r \longrightarrow \mathbb{R}$, with $n \leqslant 2r \ell^{r/2}$, for which 
\begin{equation}
\label{eq structure of CA}
C_{A^\prime} = \cap_{i \leqslant n} \{ y \in \mathbb{R}^r: \alpha_i(y) \geqslant 0\}
\end{equation} and for which for each $i \leqslant n$ there exists a subset $A_i^\prime \subset A^\prime \cap \ker \alpha_i$ with $\vert A_i^\prime\vert = r-1$ and $\ker \alpha_i = \spn(A_i^\prime)$. Therefore, using Corollary \ref{Corollary: cor to Siegel} on the $(r-1)$-by-$r$ matrix with rows given by the elements of $A_i^\prime$, without loss of generality we may assume the following: for all $i \leqslant n$, there exists a vector $x_i \in \mathbb{Z}^r \setminus \{0\}$ with $\Vert x_i\Vert_{\infty} \leqslant r^r w(A^\prime)^r$ such that for all $y \in \mathbb{R}^r$ we have $\alpha_i(y) = \langle x_i,y \rangle$. Indeed, by directly applying Corollary \ref{Corollary: cor to Siegel} we find a $z_i \in \mathbb{Z}^r \setminus \{0\}$ with $\Vert z_i\Vert_{\infty} \leqslant r^r w(A^{\prime})^r$ that is orthogonal to $\ker \alpha_i$. Hence there is some $c_i \in \mathbb{R} \setminus \{0\}$ for which $\alpha_i(y) = c_i \langle z_i,y \rangle$ for all $y \in \mathbb{R}^r$. Then $\vert c_i\vert^{-1} \alpha_i (y) = \langle \sign(c_i)z_i, y \rangle$, and without loss of generality we may rename $\vert c_i\vert^{-1} \alpha_i(y)$ as $\alpha_i(y)$ (as this preserves $C_{A^{\prime}}$) and define $x_i:=\sign(c_i)z_i$. 

We claim that for each $a^\prime \in A^\prime \setminus \{0\}$ there exists $i \leqslant n$ for which $\langle x_i, a^\prime \rangle >0$. Indeed, suppose for contradiction that there were some $a^\prime \in A^\prime \setminus \{0\}$ for which $\alpha_i(a^\prime) = 0$ for all $i$. By \eqref{eq structure of CA}, this would mean that $\lambda a^\prime \in C_{A^\prime}$ for all $\lambda \in \mathbb{R}$. Yet $0 \in \ex(H(A^\prime))$, which means that there is a non-zero linear map $\beta: \mathbb{R}^r \longrightarrow \mathbb{R}$ for which $\beta(y) >0$ for all $y \in C_{A^\prime} \setminus \{0\}$. Taking $\lambda = \pm 1$ we would have both $\beta(a^\prime) >0 $ and $\beta(-a^\prime) >0$, which gives the contradiction. Therefore for each $a^\prime \in A^\prime \setminus \{0\}$ we have $\langle a^\prime, \sum_{i \leqslant n} x_i \rangle >0$, and since these are both integer vectors we have $\langle a^\prime, \sum_{i \leqslant n} x_i \rangle \geqslant 1$.

Now suppose that $v \in \mathcal{P}(A) \setminus \{0\}$. Then $f(v) \in \mathcal{P}(A^\prime) \setminus \{0\}$. Writing \[ f(v) = a_1^\prime + \cdots + a_N^\prime\] with $a_i^\prime \in A^\prime \setminus \{0\}$, we get the inequality \[ N \leqslant \sum\limits_{j\leqslant N} \sum\limits_{i \leqslant n} \langle a_j^\prime, x_{i} \rangle = \langle f(v), \sum\limits_{i \leqslant n}  x_{i}\rangle \leqslant d\Vert f(v)\Vert_\infty \Big(\sum\limits_{i \leqslant n} \Vert x_i\Vert_\infty\Big) \leqslant \Vert f(v)\Vert_\infty 2 \ell^{r/2} r^{r+2}w(A^\prime)^r. \] Since $\Vert f(v)\Vert_\infty \leqslant d^{4d^2} w(A)^{2d^2} \Vert v\Vert_\infty$, by using the bound on $w(A^\prime)$ from \eqref{eq: changing width} we derive \[ N \leqslant \Vert v\Vert_\infty 2d^{4d^2 + 4rd^2} w(A)^{2d^2 + 3rd^2} \ell^{r/2} r^{r+2} \leqslant \Vert v\Vert_\infty 2d^{11d^3} \ell^{d} w(A)^{5d^3}. \] Writing $v = \sum_{j \leqslant N}f^{-1}(a_j^\prime)$, we have $v \in NA$ as claimed. 
\end{proof}

This completes all the necessary preparation for the first phase of the induction step.\\

\noindent \textbf{Phase 2: Quantitative details}. We will prove the following. 

\begin{Lemma}[Intersecting cones]
\label{Lemma: intersecting cones}
Let $d,d_1,d_2 \in \mathbb{Z}$, with $d \geqslant 1$ and $0 \leqslant d_1,d_2 \leqslant d$. Let $B_1,B_2 \subset \mathbb{Z}^d$ be finite sets with $\vert B_i\vert = d_i$ for each $i$, and assume that $B_1$ is linearly independent and $B_2$ is linearly independent. Let $\max_{b \in B_1 \cup B_2} \Vert b\Vert_\infty \leqslant K$ (where $K \geqslant 1$). Let $x \in \mathbb{R}^d$ and suppose $\dist(x,C_{B_1}) \leqslant X_1$ and $\dist(x,C_{B_2}) \leqslant X_2$. Then \[ \dist(x,C_{B_1} \cap C_{B_2}) \leqslant (X_1 + X_2)2^{2d} d^{10d^5}
 K^{4d^5}.\]
\end{Lemma}

First we use Siegel's lemma to construct a basis of $\mathbb{R}^d$ with certain useful properties. 

\begin{Lemma}[Basis for intersections]
\label{Lemma intersection basis} 
Let $d,d_1,d_2 \in \mathbb{Z}_{> 0}$ with $d_1,d_2 \leqslant d$. Let $B_1, B_2 \subset \mathbb{Z}^d$ be finite sets with $\vert B_i \vert = d_i$ for each $i$, and assume that $B_1$ is linearly independent and $B_2$ is linearly independent, and let $n: = \dim(\spn(B_1) \cap \spn(B_2))$. Let $\max_{b \in B_1 \cup B_2} \Vert b\Vert_\infty \leqslant K$.

Then there is a basis $V = \{v_1,\dots,v_d\}$ for $\mathbb{R}^d$ such that:
\begin{enumerate}
\item $v_i \in \mathbb{Z}^d$ for all $i$;
\item $\{v_1,\dots,v_n\}$ is a basis for $\spn(B_1) \cap \spn(B_2)$;
\item $\{v_1,\dots,v_{d_1}\}$ is a basis for $\spn(B_1)$, and $\{v_{n+1},\dots v_{d_1} \} \subset B_1$;
\item $\{v_{1},\dots,v_n, v_{d_1 + 1},\dots,v_{d_1 + d_2 - n}\}$ is a basis for $\spn(B_2)$, and $\{v_{d_1 + 1}, \dots,v_{d_1 + d_2 - n}\} \subset B_2$;
\item $\Vert v_i\Vert_\infty \leqslant d^{3d^3} K^{d^3}$ for all $i$;
\end{enumerate}
\end{Lemma}
The requirement that $\{v_{n+1},\dots,v_{d_1}\} \subset B_1$ and $\{v_{d_1 + 1},\dots,v_{d_1 + d_2 - n}\} \subset B_2$ are not vital in the application to Lemma \ref{Lemma: intersecting cones}, but will be convenient at a certain point in that proof. 
\begin{proof}
First we use Corollary \ref{Corollary: cor to Siegel} (as applied to the $d_1$-by-$d$ matrix whose rows consist of the elements of $B_1$) to construct a basis $\{X_1,\dots,X_{d-d_1}\}$ for $B_1^\perp$ consisting of vectors $X_i \in \mathbb{Z}^d$ with $\Vert X_i \Vert_\infty \leqslant (d_1!)^{1/2} d^{d/2} K^d \leqslant d^{d} K^d$. We construct a basis $\{Y_1,\dots,Y_{d-d_2}\}$ for $B_2^\perp$ in the same way.

Following this, we may construct a $(d-n)$-by-$d$ matrix $M$ whose rows are some elements of $\{X_1,\dots,X_{d-d_1},Y_1,\dots, Y_{d-d_2}\}$, where we populate the rows by choosing some $X_i$ or $Y_j$ that is not in the linear span of the rows that we have chosen so far, until we can no longer do so. By construction the rows of $M$ are a basis for $B_1^\perp + B_2^\perp$. Since $B_1^\perp + B_2^\perp = (\spn(B_1) \cap \spn(B_2))^\perp$ (by dimension counting), the rows of $M$ are also a basis for $(\spn(B_1) \cap \spn(B_2))^\perp$. Therefore applying Corollary \ref{Corollary: cor to Siegel} to the matrix $M$ we get a basis $\{v_1,\dots,v_n\}$ for $\spn(B_1) \cap \spn(B_2)$ of vectors $v_i \in \mathbb{Z}^d$ which satisfy $\Vert v_i\Vert_\infty  \leqslant (d!)^{1/2} d^{d/2} (d^d K^d)^d \leqslant d^{d^2 + d} K^{d^2}$ for each $i$. 

Now we complete $\{v_1,\dots,v_n\}$ to a basis $\{v_1,\dots,v_d\}$ for $\mathbb{R}^d$ with all the remaining properties. For $n+1 \leqslant i \leqslant d_1$, we let $v_i$ list some elements of $B_1$ that are not in $\spn(\{v_1,\dots,v_{i-1}\})$. Then for $d_1 + 1 \leqslant i \leqslant d_1 + d_2 - n$, we let $v_i$ list some elements of $B_2$ that are not in $\spn(\{v_1, \dots,v_{i-1}\})$. By dimension counting, we have that $\{v_1,\dots,v_{d_1}\}$ is a basis for $B_1$ and $\{v_1,\dots,v_n, v_{d_1 + 1}, \dots, v_{d_1 + d_2 - n}\}$ is a basis for $B_2$. We choose the remaining $ v_i$ to be integer vectors that are orthogonal to the set $\{v_j: j \leqslant d_1 + d_2 - n\}$. We can again use Corollary \ref{Corollary: cor to Siegel} to bound the norms of these $v_i$, ending up with \[ \Vert v_i\Vert_\infty \leqslant d!^{1/2} d^{d/2} (d^{d^2 + d} K^{d^2})^d \leqslant d^{d^3 + d^2 + d} K^{d^3} \leqslant d^{3d^3} K^{d^3}.\] 
This completes the lemma. 
\end{proof}

\begin{proof}[Proof of Lemma \ref{Lemma: intersecting cones}]
The proof will be by induction on $d_1 + d_2$, with the induction hypothesis being that \[ \dist(x, C_{B_1} \cap C_{B_2}) \leqslant (X_1 + X_2)2^{d_1 + d_2} d^{5(d_1 + d_2)d^4}
 K^{2(d_1 + d_2) d^4}. \] If some $d_i = 0$ then $C_{B_i} = \{0\}$ and we are done. From now on we assume that $d_1, d_2 \geqslant 1$. Since $\dist(x,C_{B_1}) \leqslant X_1$ we can write \[x = y_1 + z_1\] with $y_1 \in C_{B_1}$ and $\Vert z_1\Vert_\infty \leqslant X_1$, and similarly \[x = y_2 + z_2\] with $y_2 \in C_{B_2}$ and $\Vert z_2\Vert_\infty \leqslant X_2$. Let us emphasise that we cannot assume that $y_1,y_2 \in \mathbb{Z}^d$, nor do we currently have any control over the norms of $y_1$ or $y_2$. Both of these issues would pose difficulties were we try to induct upon the dimension $d$ by restricting to the two-dimensional subspace $\spn(\{y_1,y_2\})$. 

Let $n: = \dim(\spn(B_1) \cap \spn(B_2))$, and let $\{v_1,\dots,v_d\}$ be a basis for $\mathbb{R}^d$ that satisfies all the properties in Lemma \ref{Lemma intersection basis}. Expanding with respect to this basis, we write \[y_1 = \sum\limits_{i \leqslant n} \alpha_i v_i + \sum\limits_{n+1 \leqslant i \leqslant d_1} \beta_i v_i\] and \[y_2 = \sum\limits_{i \leqslant n} \gamma_i v_i + \sum\limits_{d_1 + 1 \leqslant i \leqslant d_1 + d_2 - n} \delta_i v_i,\] for some coefficients $\alpha_i,\beta_i, \gamma_i, \delta_i$.

We know that $\Vert y_1 - y_2\Vert_{\infty} \leqslant X_1 + X_2$, and that $\{v_1,\dots,v_d\}$ is a basis with integer coordinates and $\max_i \Vert v_i \Vert_\infty \leqslant d^{3d^3} K^{d^3}$. By Cramer's rule (or equivalently considering the change of basis matrix), we conclude that \[ \max_i (\vert \alpha _i - \gamma_i\vert, \vert \beta_i \vert, \vert \delta_i\vert) \leqslant (X_1 + X_2) d!(d^{3d^3} K^{d^3})^{d} \leqslant (X_1 + X_2) d^{4d^4} K^{d^4}.\] This implies, taking \[y_3 := \sum_{i \leqslant n} \alpha_i v_i,\] that there exists some $y_3 \in \spn(B_1) \cap \spn(B_2)$ such that 
\begin{align*}
\Vert x - y_3\Vert_\infty \leqslant \Vert z_1\Vert_\infty + \sum\limits_{n+1 \leqslant i \leqslant d_1} \vert\beta_i\vert \Vert v_i \Vert_\infty &\leqslant  (X_1 + X_2)(d^{4d^4 + 1} K^{d^4  + 1} + 1)\\
&\leqslant 2(X_1 + X_2)d^{5d^4}K^{2d^4},
\end{align*}
since $v_i \in B_1$ for all $i$ in the range $n+1 \leqslant i \leqslant d_1$. 

If $y_3 \in C_{B_1} \cap C_{B_2}$ then we are done directly from the bound on $\Vert x - y_3\Vert_\infty$. If not, let us assume without loss of generality that $y_3 \notin C_{B_1}$. The rest of the argument proceeds as follows. We know that $y_1 \in C_{B_1}$, but since $\Vert y_1 - y_3\Vert_{\infty}$ is bounded it follows that $y_1$ is nonetheless quite close to the boundary of $C_{B_1}$. Thus $y_1$ is close to $C_{B_1^\prime}$, for some $B_1^\prime \subsetneq B_1$. Hence $x$ is close to $C_{B_1^\prime}$ as well, and we may finish off by applying the induction hypothesis on the cones $C_{B_1^\prime}$ and $C_{B_2}$.

We now proceed with the details. Expanding $y_1$ in terms of the basis $B_1$, one obtains the (unique) expression \[ y_1 =\sum\limits_{b \in B_1} c_b b\] with $c_b \geqslant 0$ for all $b \in B_1$. We then claim that there must exist a set $B^\prime \subset B_1$, with $B^\prime \neq \emptyset$, for which \[c_b  \leqslant (X_1 + X_2)d^{4d^4} K^{d^4 + 1}\] for all $b \in B^\prime$. Indeed, were this not the case then $c_b > (X_1 + X_2) d^{4 d^4} K^{d^4 + 1}$ for all $b \in B_1$.  Write 
\begin{equation}
\label{eq:toexpandbothsides}
y_3 = y_1 - \sum_{ n+1 \leqslant i \leqslant d_1} \beta_i v_i
\end{equation} and recall that $\vert \beta_i\vert \leqslant (X_1 + X_2) d^{4d^4} K^{d^4}$ and $v_i \in B_1$ for each $i$ in the range $n+1 \leqslant i \leqslant d_1$. Then expand both sides of \eqref{eq:toexpandbothsides} with respect to the basis $B_1$ of $\spn(B_1)$. We get $y_3 = \sum_{b \in B_1} c^\prime_b b$, where $c^\prime_b$ is either of the form $c_b$ or $c_b - \beta_i$ for some $i$. In any case, $c^{\prime}_b \geqslant 0$ for all $b \in B_1$. So $y \in C_{B_1}$, but this is in contradiction with the earlier assumption that $y_3 \notin C_{B_1}$. 

With this set $B^\prime$, we conclude that \[ \dist(y_1, C_{B_1 \setminus B^\prime}) \leqslant (X_1 + X_2) d^{4d^4+1} K^{d^4 + 1} \leqslant (X_1 + X_2) d^{5d^4} K^{2d^4},\] and hence that \[\dist(x, C_{B_1 \setminus B^\prime}) \leqslant X_1 + (X_1 + X_2)d^{5d^4} K^{2d^4}.\] Since $\vert B_1 \setminus B^\prime \vert < d_1$, we can apply the induction hypothesis to conclude that \begin{align*}
\dist(x, C_{B_1 \setminus B^\prime} \cap C_{B_2})& \leqslant 2^{\vert B_1 \setminus B^\prime\vert + d_2}(X_1 + X_2 + (X_1 + X_2)d^{5d^4} K^{2d^4}) d^{5(\vert B_1 \setminus B^\prime\vert + d_2) d^4} K^{2(\vert B_1 \setminus B^\prime\vert + d_2) d^4}\\
&\leqslant (X_1 + X_2) 2^{d_1 + d_2}d^{5(d_1 + d_2)d^4}
 K^{2(d_1 + d_2) d^4}.
\end{align*}
Since $\dist(x, C_{B_1} \cap C_{B_2}) \leqslant \dist(x, C_{B_1 \setminus B^\prime} \cap C_{B_2})$, this closes the induction and the lemma follows. 
\end{proof}

Now let us record the precise version that we will use. 
\begin{Corollary}
\label{Corollary cone and subspace}
Let $d,d_1, d_2 \in \mathbb{Z}_{> 0}$ with $d_1,d_2 \leqslant d$, and let $K \geqslant 1$. Let $B_1 \subset \mathbb{Z}^d$ be a linearly independent set with $\vert B_1\vert = d_1$ and $\max_{b \in B} \Vert b\Vert_\infty \leqslant K$.  Let $V \leqslant \mathbb{R}^d$ be a subspace of dimension $d_2$, with a basis of $d_2$ vectors $B_2: =\{v_1,\dots,v_{d_2}\} \subset \mathbb{Z}^d \cap V$ satisfying $\Vert v_i\Vert_{\infty} \leqslant K$ for all i.

Suppose $\dist(x, C_{B_1}) \leqslant X_1$ and $\dist(x, V ) \leqslant X_2$. Then \[ \dist(x, C_{B_1} \cap V) \leqslant (X_1 + X_2) 2^{2d} d^{10 d^5} K^{4d^5}.\]
\end{Corollary}
\begin{proof}
Since $\dist(x,V) \leqslant X_2$, by replacing some vectors $v_i$ with $-v_i$ as necessary we may assume that $\dist(x, C_{B_2}) \leqslant X_2$. Then apply Lemma \ref{Lemma: intersecting cones}.
\end{proof}
Having prepared both the first and second phase of the induction step, we may plough ahead and resolve Lemma \ref{Lemma: controlling absolute B minimal}. (The third phase will be dealt with in situ.) 

\begin{proof}[Proof of Lemma \ref{Lemma: controlling absolute B minimal}]
If $B = \emptyset$ then $\Vert x\Vert_{\infty} \leqslant X$ and we are done by Lemma \ref{Lemma: controlling small elements}, so we may assume that $\vert B\vert \geqslant 1$. We then proceed by induction on $r :=\dim \spn(A)$. The base case is $r=1$. For an arbitrary non-negative real $X$, suppose $x \in S_{\abs}(A,B)$ with $\dist(x,C_B) \leqslant X$. Since $r=1$, we have moreover $x \in S_{\abs}(A,B) \subset \mathcal{P}(A) \subset C_A = C_B$, and thus in fact $\dist(x,C_B) = 0$. Observe further that $\Lambda_A \cap C_A = v \mathbb{Z}_{ \geqslant 0}$ for some non-zero vector $v \in \mathbb{Z}^d$. Taking the linear map $f:\spn(A) \rightarrow \mathbb{R}$ for which $f(v) = 1$, let $A^\prime = f(A)$, $B^\prime = f(B)$, and $x^\prime =  f(x)$. Then $w(A^\prime) \leqslant w(A)$, $\Lambda_{A^\prime} = \mathbb{Z}$, $x^\prime \in S_{\abs}(A^\prime,B^\prime)$. Applying Lemma \ref{Lemma: non simplex Davenport constant lemma}, we conclude that $x \in NA$ with $N \leqslant D(\mathbb{Z}/\Lambda_{B^\prime}) \leqslant w(A^\prime) \leqslant w(A)$. This settles the base case. 

From now on, we assume that $r \geqslant 2$ and $x \neq 0$. Our first task is to find a vector $y \in \spn(A) \setminus C_A$ for which $\Vert x-y\Vert_{\infty}$ is bounded. Indeed, choosing some $b \in B$, since $x \in \mathcal{S}_{\abs}(A,B)$ we have $x-b \notin \mathcal{P}(A)$. We know that $x \in C_A$, since $\mathcal{P}(A) \subset C_A$, and so if $x-b \notin C_A$ we let $y = x-b$ and then $\Vert x-y\Vert_{\infty} \leqslant w(A)$. Otherwise $x-b \in (\Lambda_A\cap C_A) \setminus \mathcal{P}(A) = \mathcal{E}(A)$. By Lemma \ref{Lemma:distancefromboundary}, there is therefore some $y\in \spn(A)\setminus C_A$ for which \[\Vert x-b-y\Vert_\infty \leqslant   8d^d \ell^{3d} w(A)^{3d}.\] Hence, \[ \Vert x - y\Vert_{\infty} \leqslant  8d^d \ell^{3d} w(A)^{3d} + w(A) \leqslant 16d^d \ell^{3d} w(A)^{3d}.\] 

We now apply Lemma \ref{Lemma structure of boundary} to this pair $x$ and $y$. This gives a linearly independent set $A_{\bas} \subset A$, with $\vert A_{\bas}\vert = r-1$, and a vector $z \in \spn(A_{\bas})$ for which $\Vert x- z\Vert_{\infty}\leqslant 16d^d \ell^{3d} w(A)^{3d}$. In particular 
\begin{equation}
\label{eq:distxAbas}
\dist(x, \spn(A_{\bas})) \leqslant 16d^d \ell^{3d} w(A)^{3d}.
\end{equation} We also have a vector $v \in \mathbb{Z}^d \cap \spn(A) \cap (\spn(A_{\bas}))^{\perp}$ for which $\Vert v\Vert_{\infty} \leqslant d^{2d^2} w(A)^{d^2}$ and $\langle v,u \rangle \geqslant 0$ for all $u \in C_A$. 

Phase one of the induction step is complete. We now begin the second phase, in which we show that $\dist(x, C_{B^\prime})$ is bounded for some suitable $B^\prime \subset B$. Indeed, since $\dist(x,C_B) \leqslant X$, Corollary \ref{Corollary cone and subspace} implies that\[ \dist(x, C_B \cap \spn(A_{\bas})) \leqslant (16d^d \ell^{3d} w(A)^{3d} + X)2^{2d} d^{10 d^5} w(A)^{4d^5}.\] Let $B^\prime = B \cap \spn(A_{\bas})$. We then have $C_B \cap \spn(A_{\bas}) = C_{B^\prime}$. To justify this assertion, note that if $u \in C_B \cap \spn(A_{\bas})$ we have $u = \sum_{b \in B} c_b b$ for some coefficients $c_b  \geqslant 0$. But then \[0 = \langle v,u\rangle = \sum_{b \in B} c_b \langle v,b\rangle = \sum_{b \in B \setminus B^\prime} c_b \langle v, b\rangle\] since $v \in \spn(A_{\bas})^{\perp}$. As $\langle v, b \rangle >0$ for all $b \in B \setminus B^\prime$ we must have $c_b = 0$ for all $b \in B \setminus B^\prime$. Hence $y \in C_{B^\prime}$. (The reverse inclusion $C_{B^\prime} \subset C_B \cap \spn(A_{\bas})$ is immediate from definitions.) Therefore, 
\begin{equation}
\label{eq: dist x CBprime}
\dist(x, C_{B^\prime}) \leqslant (16d^d \ell^{3d} w(A)^{3d} + X)2^{2d} d^{10 d^5} w(A)^{4d^5}.
\end{equation}
 
Now we move onto the third phase of the induction step. Let $A^\prime = A \cap \spn(A_{\bas})$. We now collect a few facts about $A^\prime$ and about $x$. Firstly, if $a \in A \setminus A^\prime$ then $\langle v , a \rangle >0$, and thus $\langle v, a \rangle \geqslant 1$ as both $v$ and $a$ are in $\mathbb{Z}^d$. Next, letting $x_0$ be the orthogonal projection of $x$ onto $\spn(A_{\bas})$, we have \[ \langle v ,x \rangle = \langle v, x - x_0 \rangle \leqslant d \Vert v\Vert_{\infty}\dist(x, \spn(A_{\bas})).\] Finally, since $x \neq 0$, we may write $x = a_1 + \dots + a_N$ for some $a_i \in A \setminus (B \cup \{0\})$. Putting everything together we then have 
\begin{align}
\label{eq number of things not in the kernel}
\vert \{ i \leqslant N: a_i \in A \setminus A^\prime\} \vert \leqslant  \sum\limits_{i=1}^N \langle v, a_i \rangle &= \langle v,x\rangle\nonumber \\
& \leqslant d \Vert v\Vert_{\infty}\dist(x, \spn(A_{\bas})) \nonumber\\
& \leqslant 16d^{4d^2} \ell^{3d} w(A)^{4d^2}.
\end{align}

Now define \[x^\prime: = \sum_{i \leqslant N: a_i \in A^\prime} a_i \in \mathcal{P}(A^\prime).\] Then 
\begin{equation}
\label{dist between x and xprime}
\Vert x - x^\prime\Vert_{\infty} \leqslant 16d^{4d^2} \ell^{3d} w(A)^{4d^2 + 1} \leqslant 16d^{4d^2} \ell^{3d} w(A)^{5d^2},
\end{equation} and so 
\begin{align}
\label{eq:disttobprime}
\dist(x^\prime, C_{B^\prime}) &\leqslant \dist(x, C_{B^\prime}) + \Vert x - x^\prime\Vert_{\infty} \nonumber \\
&\leqslant (16d^d \ell^{3d} w(A)^{3d} + X)2^{2d} d^{10 d^5} w(A)^{4d^5} + 16d^{4d^2} \ell^{3d} w(A)^{5d^2} \nonumber \\
&\leqslant (X+1) 2^{7d} d^{11 d^5} \ell^{3d} w(A)^{7d^5}.
\end{align}
\noindent What's more, $x^\prime \in \mathcal{S}_{\abs}(A^\prime,B^\prime)$. Indeed, $x^\prime \in \mathcal{P}(A^\prime)$ by construction, and if $x^\prime - b^\prime \in \mathcal{P}(A^\prime)$ for some $b^\prime \in B^\prime$ then $x - b^\prime \in \mathcal{P}(A)$, in contradiction to the assumption that $x \in \mathcal{S}_{\abs}(A,B)$. 

We now apply the induction hypothesis to the sets $A^\prime$ and $B^\prime$, and to the element $x^\prime$. The hypotheses are satisfied (taking $\dist(x^\prime, C_{B^\prime})$ from \eqref{eq:disttobprime}), since $B^\prime$ is linearly independent (though possibly empty), and $0 \in \ex(H(A^\prime))$; this is since, if $V \cap \spn(A)$ is a separating hyperplane for $H(A)$ with $V \cap H(A) = \{0\}$, then $V \cap \spn(A^\prime)$ is a separating hyperplane for $H(A^\prime)$ with $V \cap H(A^\prime) = 0$. 

So, $x^\prime \in N^\prime A^\prime$ for some 
\begin{align*}
N^\prime &\leqslant ((X+1) 2^{7d} d^{11 d^5} \ell^{3d} w(A)^{7d^5} + 1) 2^{10d(r-1)} d^{11d^5 (r-1)} \ell^{3d(r-1)} w(A)^{7d^5 (r-1)}\\
&\leqslant (X+1) 2^{10dr - 3d + 1} d^{11d^5 r} \ell^{3dr} w(A)^{7d^5 r}.
\end{align*} 
\noindent Finally, adding in the contribution from \eqref{eq number of things not in the kernel} from those $a \in A \setminus A^\prime$, we deduce that $x \in N A$ for 
\begin{align*}
 N &\leqslant (X+1) 2^{10dr - 3d + 1} d^{11d^5 r} \ell^{3dr} w(A)^{7d^5 r} + 16d^{4d^2} \ell^{3d} w(A)^{4d^2} \\
 &\leqslant (X+1)   2^{10dr} d^{11d^5 r} \ell^{3dr} w(A)^{7d^5 r}
 \end{align*} as $d \geqslant 2$. This completes the induction, and the lemma is proved.
\end{proof}

So Lemma \ref{Lemma: controlling absolute B minimal} is settled, and with it our main effective structure result, Theorem \ref{thm main theorem}. \\

 \appendix
 \section{Convex sets}
 \label{Appendix convex sets}
In this appendix we collect together some standard facts about convex polytopes (i.e. convex hulls of finite subsets of Euclidean space). Our main references will be \cite{Br83} and \cite{Zi95}.

\begin{Lemma}[Extremal points]
\label{Lemma Extremal points}
Let $A  \subset \mathbb{R}^d$ be a finite set. Then $\ex(H(A)) \subset A$ and $H(A) = H(\ex(H(A))$. 
\end{Lemma}

\begin{proof}
This is \cite[Theorem 7.2]{Br83}. 
\end{proof}

\begin{Lemma}[Structure of $H(A)$]
\label{Lemma structure of convex hull}
Let $A \subset \mathbb{R}^d$ be a finite set with $\vert A\vert = \ell$, $0 \in \ex(H(A))$, and assume that $\spn(A) = \mathbb{R}^d$. Then there is a finite collection of  maps $\{\alpha_1, \dots, \alpha_n\} \in (\mathbb{R}^d)^*$ and constants $\{c_1, \dots, c_n\} \subset \mathbb{R}_{ \leqslant 0}$ for which $n \leqslant 2d\ell^{d/2}$ and 
\begin{enumerate}
\item $H(A) = \cap_{i \leqslant n}\{x \in \mathbb{R}^d: \alpha_i(x) \geqslant c_i\}$;
\item $C_A = \cap_{i \leqslant n:\, c_i = 0} \{ x \in \mathbb{R}^d: \alpha_i(x) \geqslant 0\}$;
\item $\partial(C_A) = \cup_{i \leqslant n:\, c_i = 0} (\ker \alpha_i \cap C_A)$
\item for all $i$ such that $c_i = 0$, there exists a set $A_i \subset A \setminus \{0\}$ with $\vert A_i \vert = d-1$ and $\spn(A_i) = \ker \alpha_i$. 
\end{enumerate}
\end{Lemma}
\begin{proof}
Part (1) is the fundamental theorem of convex polytopes, given as \cite[Theorem 9.2]{Br83}. To prove Part (2), we note that $c_i \leqslant 0$ for all $i$, since $0 \in A \subset H(A)$. Then \[ C_A = \cup_{N \geqslant 1} NH(A) = \cup_{N \geqslant 1} \cap_{i \leqslant n} \{x \in \mathbb{R}^d: \alpha_i(x) \geqslant Nc_i\} =  \cap_{i \leqslant n:\, c_i = 0} \{x \in \mathbb{R}^d: \alpha_i(x) \geqslant 0\}\] as claimed. Part (3) follows immediately from part (2) (see \cite[Theorem 8.2 (a)]{Br83}).

For Part (4) we appeal to \cite[Theorem 8.2 (c)]{Br83}, assuming as we may that the expression $H(A) = \cap_{i \leqslant n}\{x \in \mathbb{R}^d: \alpha_i(x) \geqslant c_i\}$ is irreducible (i.e. the collection of maps $\{ \alpha_1, \dots, \alpha_n\}$ cannot be replaced with a proper subset). This result tells us that \[\{x \in \mathbb{R}^d: \alpha_i(x) = c_i\} \cap H(A)\] is a facet of $H(A)$, i.e. is a $d-1$ dimensional face. Now, if $F$ is a facet of $H(A)$, \cite[Theorem 7.2]{Br83} and \cite[Theorem 7.3]{Br83} imply that $F$ is a polytope, $F = H(\ex(F))$, and $\ex(F) \subset \ex(H(A)) \subset A$.   Therefore, if $c_i = 0$ we see that $\ker \alpha_i \subset \spn(A \cap \ker \alpha_i)$. Since every spanning set contains a basis we may find the set $A_i$ as claimed in (4). 

Regarding the claim that $n \leqslant 2 d\ell^{d/2}$, this bound follows from the celebrated Upper Bound Theorem of McMullen (\cite{Mc70}, or \cite[Theorem 8.23]{Zi95} of Ziegler's textbook), which gives a tight upper bound for the number of facets of a convex polytope. This is since the pair $(\alpha_i,c_i)$ is determined (up to scalar multiples) by the facet $F: = \{ x \in \mathbb{R}^d: \alpha_i(x) = c_i\} \cap H(A)$. However, one doesn't need the full strength of the Upper Bound Theorem to get the order-of-magnitude bound $2d \ell^{d/2}$; one could instead use the easier argument of Seidel \cite{Se95}, summarised in the remark before Section 8.5 of \cite{Zi95}. This bounds the number of facets by $2\sum_{i \leqslant d/2} (\begin{smallmatrix} \ell \\ d\end{smallmatrix})$, which is trivially at most $2d \ell^{d/2}$. 
\end{proof}

To aid the reader seeking the references in \cite{Br83}, we should say that \cite{Br83} uses the symbol $H$ differently; there, $H$ denotes a hyperplane, whereas for us $H$ is the convex hull.

 \bibliographystyle{plain}
 
 \bibliography{frobupdate}
   \end{document}